\documentclass{article}
\usepackage{style}
\usepackage[utf8]{inputenc}

\usepackage{caption}
\usepackage{subcaption}

\usepackage{xspace,prettyref}
\usepackage{dsfont}
\usepackage{placeins}

\newrefformat{eq}{(\ref{#1})}
\newrefformat{chap}{Chapter~\ref{#1}}
\newrefformat{sec}{Section~\ref{#1}}
\newrefformat{alg}{Algorithm~\ref{#1}}
\newrefformat{fig}{Figure~\ref{#1}}
\newrefformat{tab}{Table~\ref{#1}}
\newrefformat{rmk}{Remark~\ref{#1}}
\newrefformat{clm}{Claim~\ref{#1}}
\newrefformat{def}{Definition~\ref{#1}}
\newrefformat{cor}{Corollary~\ref{#1}}
\newrefformat{lmm}{Lemma~\ref{#1}}
\newrefformat{prop}{Proposition~\ref{#1}}
\newrefformat{prob}{Problem~\ref{#1}}
\newrefformat{app}{Appendix~\ref{#1}}
\newrefformat{hyp}{Hypothesis~\ref{#1}}
\newrefformat{thm}{Theorem~\ref{#1}}

\usepackage{graphicx}  

\begin{document}
\title{Phase Transitions in Planted $k$-Factor Recovery}
 \author{Julia Gaudio\footnote{Northwestern University, Department of Industrial Engineering and Management Sciences; julia.gaudio@northwestern.edu} \and Colin Sandon \footnote{\'Ecole Polytechnique F\'ed\'erale de Lausanne; colin.sandon@epfl.ch} \and Jiaming Xu \footnote{Duke University, The Fuqua School of Business; jiaming.xu868@duke.edu} \and Dana Yang \footnote{Cornell University, Department of Statistics and Data Science; xy374@cornell.edu}}
\date{\today}

\maketitle

\begin{abstract}
This paper studies the problem of inferring a $k$-factor, specifically a spanning $k$-regular graph, planted within an \ER random graph $\calG(n,\lambda/n)$. We show that as the average degree $\lambda$ surpasses the critical threshold of $1/k$, the inference problem undergoes a transition from almost exact recovery  to partial recovery. Moreover, as $\lambda$ tends to infinity, the accuracy of recovery diminishes to zero. In addition, we characterize the recovery accuracy of a linear-time iterative pruning algorithm and show that it achieves almost exact recovery when $\lambda < 1/k$. A key component of our analysis is a two-step cycle construction: we first build trees through local neighborhood exploration and then connect them by sprinkling using reserved edges. Interestingly, for proving impossibility of almost exact recovery, we construct $\Theta(n)$ many small trees of size $\Theta(1)$, whereas for establishing the algorithmic lower bound, a single large tree of size $\Theta(\sqrt{n\log n})$ suffices.
\end{abstract}

\section{Introduction}
This paper studies the following planted subgraph recovery problem. 
We first generate a background \ER random graph $G_0 \sim \calG(n,p)$,
with $n$ vertices each pair of which are independently connected with probability $p$. A subset $S \subset [n]$ of size $m$  is selected uniformly at random. From a given family $\calH$ of labeled graphs with vertex set $S$, $H^*$ is chosen uniformly at random and embedded into $G_0$ by adding its edges. 
Letting $G$ denote the resulting graph, the goal is to recover the hidden subgraph $H^*$ based on the observation of $G.$ 
Depending on the choice of $\calH$, this framework encompasses a wide range of planted subgraph problems, including the model of planted clique~\cite{kuvcera1995expected}, tree~\cite{massoulie2019planting}, Hamiltonian cycle~\cite{bagaria2020hidden}, matching~\cite{moharrami2021planted},  nearest-neighbor graph~\cite{ding2020consistent}
and many others. 

Our study is motivated by the following fundamental question: in which regime in terms of $(n,m,p)$, can we recover the hidden subgraph $H^*$? Specifically, let $\hat H \equiv \hat H(G)$ denote an estimator of $H^*$ that is a set of edges on $K_{n}$, the complete graph on $n$ vertices.
The reconstruction error is 
\begin{equation}
\risk(H^*, \hat H) = \frac{\left| \Mplanted \symdiff \hat{H} \right|}{|\Mplanted|} ,
\label{eq:risk}
\end{equation}
where $\symdiff$ denotes the symmetric set difference. We say $\hat{H}$
achieves exact, almost exact, or partial recovery, if with high probability $\risk(H^*,\hat{H})$ is $0$, $o(1)$, or $1-\Omega(1)$, respectively.\footnote{Since the trivial estimator $\hat{H}=\emptyset$ has reconstruction error equal $1$, we require the partial recovery to achieve reconstruction error strictly bounded away from $1.$ Moreover,  due to the equivalence between Hamming error and mean-squared error (See Appendix~\ref{sec:equivalence}), the main results of this paper also hold under the mean-squared error metric.} 
Interestingly, for certain choices of $\calH$, the problem exhibits a peculiar ``all-or-nothing'' (AoN) phase transition in the asymptotic regime $n \to \infty$: the minimum reconstruction error, namely $\inf_{\hat{H}} \mathbb{E}_{G}[\ell(H^*,\hat{H}(G))],$  falls sharply from 1 to 0 at a certain critical threshold. For example, when $\calH$ consists of $m$-cliques and $p=1/2$, the problem reduces to the planted clique problem, which has AoN phase transition at threshold $m^*=2\log_2(n).$
As another example, when $\calH$ consists of $m$-paths and $p=\lambda/n$ for a constant $\lambda,$ then we arrive at the planted path problem,  which has AoN phase transition at threshold $m^*=\log(n)/\log(1/\lambda)$~\cite{massoulie2019planting}. 
More interestingly, the critical threshold coincides with the so-called first-moment threshold, at which the expected number of copies of subgraphs in $\calH$ contained in the background graph $G_0$ is approximately $1.$

The work of~\cite{Mossel2023} established the AoN phase transition for significantly broader classes of graph families  $\calH.$ Loosely speaking, the planted subgraph recovery model exhibits AoN at the first moment threshold when the hidden graph $H^*$ is either sufficiently dense
and balanced ($H^*$ has the maximal edge density among all its subgraphs) or small and strictly balanced. \footnote{\cite[Theorem 2.5]{Mossel2023} provides a more general necessary-and-sufficient characterization of AoN for sufficiently dense graphs, at the so-called ``generalized expectation threshold.''} 
Notably, AoN has also been established in various other high-dimensional inference problems, including sparse linear regression~\cite{reeves2021all}, sparse tensor PCA~\cite{niles2020all}, group testing~\cite{truong2020all,coja2022statistical}, graph alignment~\cite{wu2022settling}, and others. Despite these significant advancements, an interesting question remains elusive: \emph{what is the underlying reason for the onset of AoN?}

In this paper, we consider a complementary direction, namely the case of large, sparse graphs.
Specifically, we assume $\calH^*$ consists of all $k$-factors, the $k$-regular graphs spanning the vertex set $[n]$, where $k\ge 1$ is a fixed integer. This is known as the planted $k$-factor model~\cite{sicuro2021planted}. When $k=1$, this reduces to the planted matching problem~\cite{moharrami2021planted}. We assume $p=\lambda/n$, where $\lambda$ may scale with $n.$ Let $\mu_G$ denote the posterior distribution:
\begin{align}
\mu_G(H) \triangleq \prob{H^*=H \mid G } = \frac{1}{|\calH(G)|} \indc{H \in \calH(G)}, \quad \forall H \in \calH,  \label{eq:posterior_distribution}
\end{align}
which is simply a uniform distribution supported on the set of $k$-factors in $G$,
denoted by $\calH(G)$. 
Clearly, in the extreme case of $\lambda=0$, 
$\calH(G)=\{H^*\}$ is a singleton and $\mu_G$ is a delta measure supported on $H^*$.  As $\lambda$ increases, $\mu_G$ spreads over a larger subset of $k$-factors. 

A bit more precisely, the expected number of $k$-factors $H$ with $|H^*\triangle H|=2t$ is roughly
\begin{align}
&\binom{kn/2}{t} \cdot (2t-1)!! \cdot (\lambda/n)^t
\approx (k\lambda)^t \left(1-\frac{2}{kn}\right)\left(1-\frac{4}{kn}\right) \cdots \left(1-\frac{2(t-1)}{kn}\right),
\label{eq:mean_bound}
\end{align}
where $\binom{kn/2}{t}$ counts the ways to select $t$ edges from $H^*;$
$(2t-1)!!$ counts the number of pairings among the $2t$ endpoints of these edges\footnote{Note that some pairings may not generate valid $k$-factors;
thus, the LHS of \prettyref{eq:mean_bound} is only an upper bound as shown in~\prettyref{lmm:enumeration}. 
} ; and $(\lambda/n)^t$ is the probability that all $t$ pairs of endpoints are connected in $G$. (See~\prettyref{lmm:enumeration} for more details).

 This suggests that if $\lambda=o(1)$ there will be no $k$-factor other than the planted one in $G$; if $\lambda$ is a constant less than $1/k$ then there will be some constant number of $k$ factors that are all very close to the planted one; and that for $\lambda>1/k$ there will be a large number of $k$-factors in the graph, most of them differing from the planted one by $t^*(\lambda)$ edges, for some function $t^*(\lambda)\approx (1-\frac{1}{k\lambda})(kn/2)$ that approximately maximizes~\prettyref{eq:mean_bound}. 
 This behavior indicates a gradual decline in the reconstruction accuracy, and hence the absence of an AoN phase transition.
 While the intuitive argument above is straightforward, formalizing it rigorously is highly non-trivial and constitutes the main contribution of this paper. 
 
\begin{theorem}
Consider the planted $k$-factor model with $n$ nodes and $p=\lambda/n$. The following hold with probabilities tending to $1$ as $n \to \infty$:
\begin{itemize}
\item (Exact recovery) If $\lambda=o(1)$, then $\mu_G$ is a delta measure on $H^*$ and the minimum reconstruction error is $0$;
\item (Almost exact recovery) If $\Omega(1) \le \lambda \le 1/k$, then $1- o(1)$ of the probability mass of $\mu_G$ is supported on $k$-factors that differ from $H^*$ by $o(1)$ fraction of their edges and the minimum  reconstruction error is $o(1)$ ;
\item  (Partial recovery) If $\lambda>1/k$ is a constant, then 
$1-o(1)$ of the probability mass of $\mu_G$ is supported on $k$-factors which share $[\Omega(1),1-\Omega(1)]$ fraction of their edges with $H^*$, and the minimum reconstruction error is within $[\Omega(1), 1-\Omega(1)]$;
\item (Nothing) If $\lambda=\omega(1)$, then  $1-o(1)$ of the probability mass of $\mu_G$ is supported on $k$-factors which share $o(1)$ fraction of their edges with $H^*$, and the minimum reconstruction error is $1-o(1)$. 
\end{itemize}
\end{theorem}
In particular, the problem transitions from almost exact recovery to partial recovery at the sharp threshold of $\lambda=1/k.$ This not only recovers the known result for $k=1$~\cite{Ding2023} but also resolves the conjecture posed in~\cite{sicuro2021planted} for $k \ge 2$ in the positive. 
In comparison, it is well-known that the first-moment threshold for $k$-factors in the background graph $G_0$ is $\lambda_{\rm 1M}=e(k!)^{2/k} /k$ (cf.~\cite[Corollary 2.17]{bollobas2001random}) and the critical threshold for the existence of $k$-factors in $G_0$ is $\lambda_c=\log n +(k-1)\log\log n + \omega(1). $ 
Intuitively, the almost exact recovery threshold is lower than the first-moment threshold because a single planted $k$-factor, together with the edges in the background Erd\H{o}s--R\'enyi graph $G_0$, can generate many spurious $k$-factors—even when the expected number of $k$-factors in $G_0$ is small. See Fig.~\ref{fig:phase_diagram} for a graphical illustration of the different thresholds. 
\begin{figure}[t]
\begin{center}
\begin{tikzpicture}[scale=0.8]
    \draw[->] (0,0) -- (16,0) node[right] {$\lambda$};
    \draw[black, thick] (2,1.5) circle (1);
    \draw[red, thick,fill] (2,1.5) circle (0.05) node[above] {$H^*$};
    \draw[black, thick] (6,1.5) circle (1);
        \fill[black!30, opacity=0.5] (6,1.5) circle (0.2);
    \draw[red, thick,fill] (6,1.5) circle (0.05) node[above] {$H^*$};
    \draw[black, thick] (10,1.5) circle (1);
    \fill[black!30,opacity=0.5,even odd rule, ] (10,1.5) circle (0.6) (10,1.5) circle (0.4);
    \draw[red, thick,fill] (10,1.5) circle (0.05) node[above] {$H^*$};
     \draw[black, thick] (14,1.5) circle (1);
    \fill[black!30,opacity=0.5,even odd rule, ] (14,1.5) circle (1) (14,1.5) circle (0.8);
    \draw[red, thick,fill] (14,1.5) circle (0.05) node[above] {$H^*$};

    \node[below] at (0,0) {0};
    \node[below] at (4,0) {$o(1)$};
    \draw[thick] (4,0) -- (4,0.1);
    \node[below] at (8,0) {$1/k$};
     \draw[thick,dashed] (8,0) -- (8,3);
    \node[below] at (12,0) {$\omega(1)$};
    \node[below] at (10,0){$\lambda_{\rm 1M}$};
    \draw[thick] (10,0) -- (10,0.1);
     \draw[thick,dashed] (12,0) -- (12,3);
     \node[below] at (14,0){$\lambda_{c}$};
     \draw[thick] (14,0) -- (14,0.1);

 \node[below] at (2,-0.5) {Exact};
 \node[below] at (6,-0.5) {Almost exact};
  \node[below] at (10,-0.5) {Partial};
\end{tikzpicture}
\caption{Cartoon plot of the phase diagram with varying $\lambda$: The circle represents the space of all possible $k$-factors $H$ centered at the hidden one $H^*$ according to the Hamming distance $|H\Delta H^*|$; the grey area contains almost all $k$ factors in the observed graph $G.$}
\label{fig:phase_diagram}
\end{center}
\end{figure}
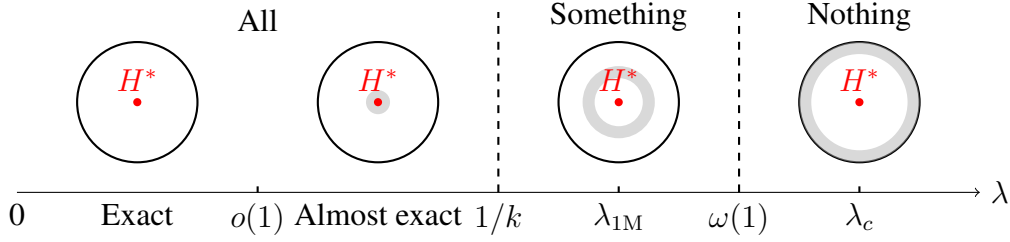

Complementing the study of the information-theoretic thresholds, we also investigate the algorithmic thresholds. We show that the three recovery thresholds can be achieved efficiently via an iterative pruning algorithm proposed in~\cite{sicuro2021planted}.
\begin{theorem}
There exists a linear-time $O(|E(G)|)$ iterative pruning algorithm that achieves exact recovery, almost exact recovery, and partial recovery, when $\lambda=o(1),$ 
$\lambda \le 1/k$,
and $\lambda=O(1),$ respectively. 
\end{theorem}
 Note that for $k\ge 2$, the set of $k$-factors may contain different isomorphism classes. This slightly differs from the setup in~\cite{Mossel2023} which assumes $\calH$ contains only a single isomorphism class. For example, a $2$-factor corresponds to a disjoint union of cycles with total length $n$, and an isomorphism class corresponds to a cycle length configuration. If we restrict $\calH$ to contain only Hamiltonian cycles, we arrive at the planted Hamiltonian cycle~\cite{bagaria2020hidden}. Our results continue to hold for the planted Hamiltonian cycle model via a reduction argument (See~\prettyref{sec:Hamiltonian} for details).

Finally, we briefly comment on the closely related detection problem.  If we test the planted $k$-factor model against the \ER random graph $\calG(n,\lambda/n)$, we can easily distinguish the hypotheses by counting the total number of edges. If we instead test against $\calG(n, \lambda'/n)$, where $\lambda'$ is chosen so that the average number of edges matches between the planted and null model:
$
\frac{kn}{2} + \left( \binom{n}{2} - \frac{kn}{2} \right) \frac{\lambda}{n}
= \binom{n}{2} \frac{\lambda'}{n},
$
then we can still test the hypothesis based on the minimum degree or the existence of a $k$-factor. The test based on the existence of a $k$-factor succeeds as long as $\lambda \le \log n + (k-1) \log \log n - \omega(1)$, as the null model does not contain any $k$-factors with high probability. In summary, we see that detection is much easier than recovery for the planted $k$-factor problem.

The rest of the paper is organized as follows. In \prettyref{sec:main}, we formally state our main results. \prettyref{sec:proof_ideas} outlines the key proof ideas.  \prettyref{sec:conclusions} concludes the paper with open questions. Formal proofs and additional details are deferred to the appendices.

 \section{Main Results}\label{sec:main}

\subsection{Information-theoretic Limits}
The following result shows that the exact recovery of $H^*$ is information-theoretically possible if and only if $\lambda=o(1). $

 \begin{theorem}[Exact recovery, positive and negative ]\label{thm:exact_recovery}
Consider the planted $k$-factor model conditioning on $H^*$. If $\lambda=o(1)$, then $H^*$ is the unique $k$-factor in graph $G$ with probability $1-o(1)$.  Conversely, if $\lambda=\Omega(1)$, then $G$ contains a $k$-factor $H \neq H^*$ with probability $\Omega(1).$ 
\end{theorem}

The following two theorems together show that the almost exact recovery of $H^*$ is possible if and only if $\lambda \le 1/k.$

\begin{theorem}[Almost exact recovery, positive]\label{thm:almost_perfect}
Suppose that
\begin{align}
    \lambda k \le 1+\epsilon \label{eq:possibility}
\end{align}
for some $\epsilon \in [0,1)$.  Let 
\begin{align}
\beta =\max\left\{4 \log(1+\epsilon), \sqrt{\frac{8\log n}{n}}\right\} \label{eq:beta_def}
\end{align}
and 
$\widehat{H}$ denote an estimator that outputs a $k$-factor in $G$.
Then 
$$
\mathbb{P}\left\{\ell\left(\widehat{H},H^*\right) \ge 2\beta  \mid H^* \right\}
\le e^{1/2} \beta, 
$$
and moreover,
$
\expect{\risk( \widehat{H}, H^*) \mid H^*}\le 
6\beta.$
\end{theorem}
Setting $\epsilon=0$ in~\prettyref{thm:almost_perfect}
shows that almost eact recovery is achievable when $\lambda k \le 1.$
\begin{theorem}[Almost exact recovery, negative]\label{thm:impossibility}
If 
\begin{align}
    \lambda k \ge 1+\epsilon \label{eq:impossibility}
\end{align}
for some $\epsilon>0$, then there exists a constant $\epsilon'> 0$ depending only on $\epsilon$ and  $k$ such that
 for any estimator $\widehat{H}$ and large $n$, 
 \[\mathbb{P}\left(\ell(H^*, \hat{H}) \geq \epsilon' \right) = 1 -e^{-\Omega(n)}.\]
 It follows that for large enough $n$,
$$
\expect{\risk(\hat{H}, H^*)}\ge \frac{\epsilon'}{2}.
$$

\end{theorem}

Next, we move to partial recovery when $\lambda=\Omega(1).$  Let $\hat{H}$ be the following estimator: 
\begin{align*}
\hat{H}(u,v) &=
\begin{cases}
1 & \deg(u) = k \text{ or } \deg(v) = k\\
0 & \text{otherwise}.
\end{cases}
\end{align*}
As we will see later, $\hat{H}$ coincides with the initial step of the iterative pruning algorithm.  The following lemma shows that $\hat{H}$ achieves partial recovery when $\lambda=O(1). $

\begin{theorem}[Partial recovery, positive]
\label{thm:partial}
 Under the planted $k$-factor model with $p=\lambda/n$,   
$$
\mathbb{E}[\ell(\hat{H}, H^*) \mid H^* ] \leq
1-e^{-2\lambda}
$$
and 
$$
\prob{ \ell(\hat{H}, H^*) \le 1- \frac{1}{2} e^{-2\lambda}\mid H^* } \ge 1-O(1/n).
$$

Furthermore, with probability $1-O(1/n)$, 
all $k$-factors in $G$
agree on at least $e^{-2\lambda} n k/4$ edges
and in particular, share at least a $e^{-2\lambda}/2$ fraction of their edges with $H^*$. 
\end{theorem}

The following result derives a complementary lower bound on the error of any estimator, proving that partial recovery is impossible when $\lambda=\omega(1). $

\begin{theorem}[Partial recovery, negative]\label{thm:nothing}
Under the planted $k$-factor model with $p=\lambda/n$, there exists a universal constant $0<c<1$ such that  for any estimator $\hat{H}$, 
$$
\prob{ \ell(\hat{H}, H^*) \ge 1-  \frac{\log(k^2/c)}{\log(k\lambda)} } \ge 1- 3 (k^2/c)^{-nk/4}.
$$
It follows that 
$$
\mathbb{E}[\ell(\hat{H}, H^*)]
\ge 1- \frac{\log(k^2/c)}{\log(k\lambda)}  - 3 (k^2/c)^{-nk/4}.
$$
Moreover, with probability at least $1- 3(k^2/c)^{-nk/8}$, at least a $1-(k^2/c)^{-nk/8}$
fraction of $k$-factors in graph $G$ share at most $\frac{2\log(k^2/c)}{\log(k\lambda)} $ fraction of their edges with $H^*$.
\end{theorem}

\subsection{Algorithmic Limits}
In the previous subsection, we have been focusing on characterizing the information-theoretic thresholds. In this subsection, we explore the algorithmic limits.

Theorems \ref{thm:exact_recovery} and \ref{thm:almost_perfect} imply that to achieve either the exact or almost exact recovery thresholds, it suffices to output any $k$-factor in the observed graph $G$. It is known that finding a $k$-factor in general graphs can be done efficiently in total time $O(n^3 k)$~\cite{meijer2009algorithm}.
Alternatively, for the planted $k$-factor model, we can show a linear-time iterative pruning algorithm~\cite{sicuro2021planted}
outputs a set of edges $\hat{H}$ (which may not necessarily be a valid $k$-factor) that 
achieves the thresholds for the exact, almost exact, and partial recovery of $H^*.$

\subsubsection{Iterative Pruning algorithm} To begin with, each vertex $i$ is assigned an initial capacity $\kappa_i = k$. The capacity of each vertex will keep track of the number of unidentified planted edges incident to $i.$ Then we repeatedly apply the following pruning procedure until all vertices have degrees bigger than their capacities. Find a vertex whose degree equals its capacity. Note that all its incident edges must be planted. Thus we
remove this vertex and all its incident edges from the graph $G$, and decrease by $1$ the capacities of the endpoints of the removed edges. If there exist vertices whose capacities drop to $0$, then their incident edges must be unplanted. Thus we remove these vertices, together with all their incident edges. Finally, when the iterative pruning process stops, we output $\hat{H}$ to be the set of planted edges identified during the process.

\begin{figure}[t]
\centering
\begin{minipage}{.45\linewidth}
  \centering
\includegraphics[width=\linewidth]{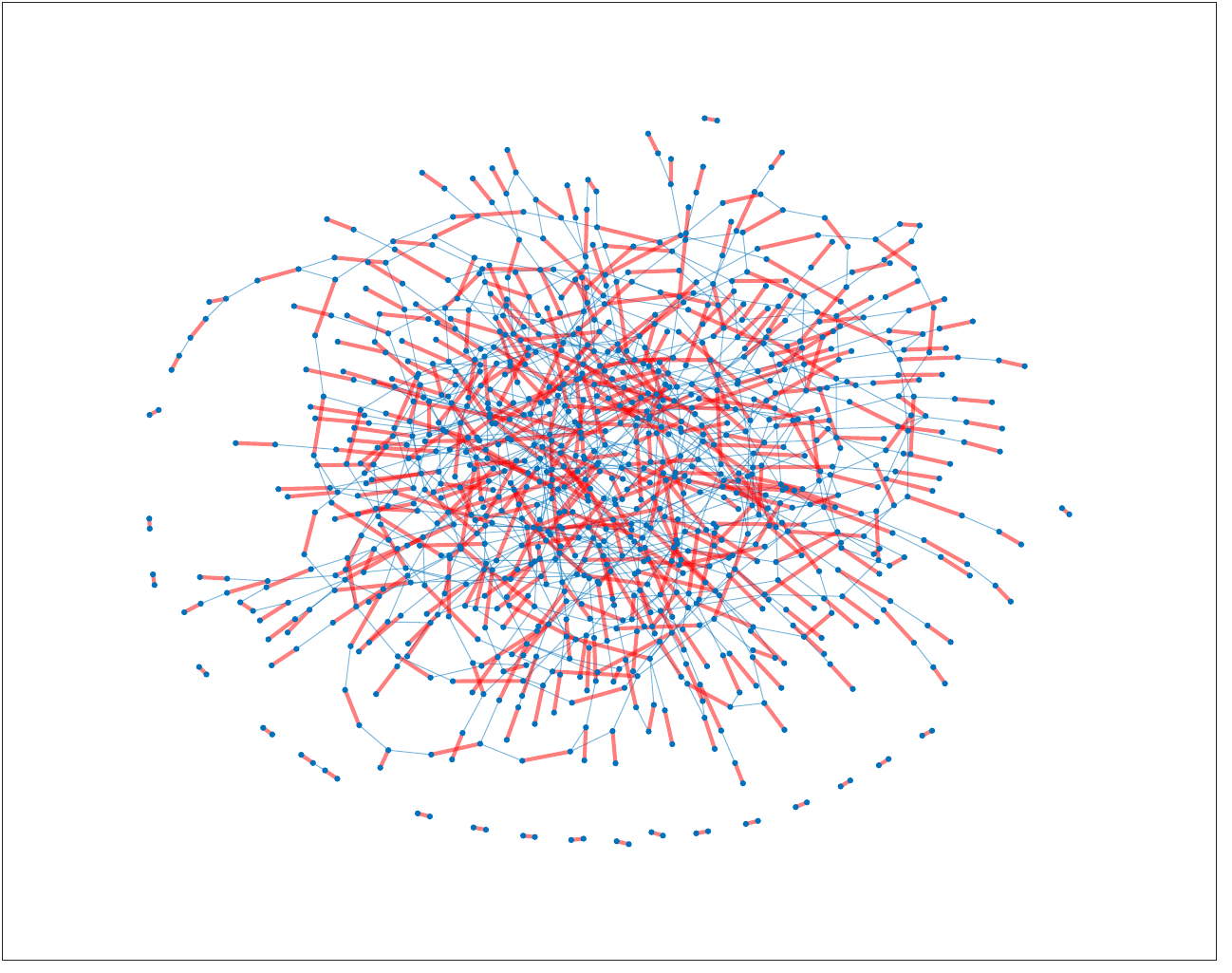}
\end{minipage}
\begin{minipage}{.45\linewidth}
  \centering  \includegraphics[width=\linewidth]{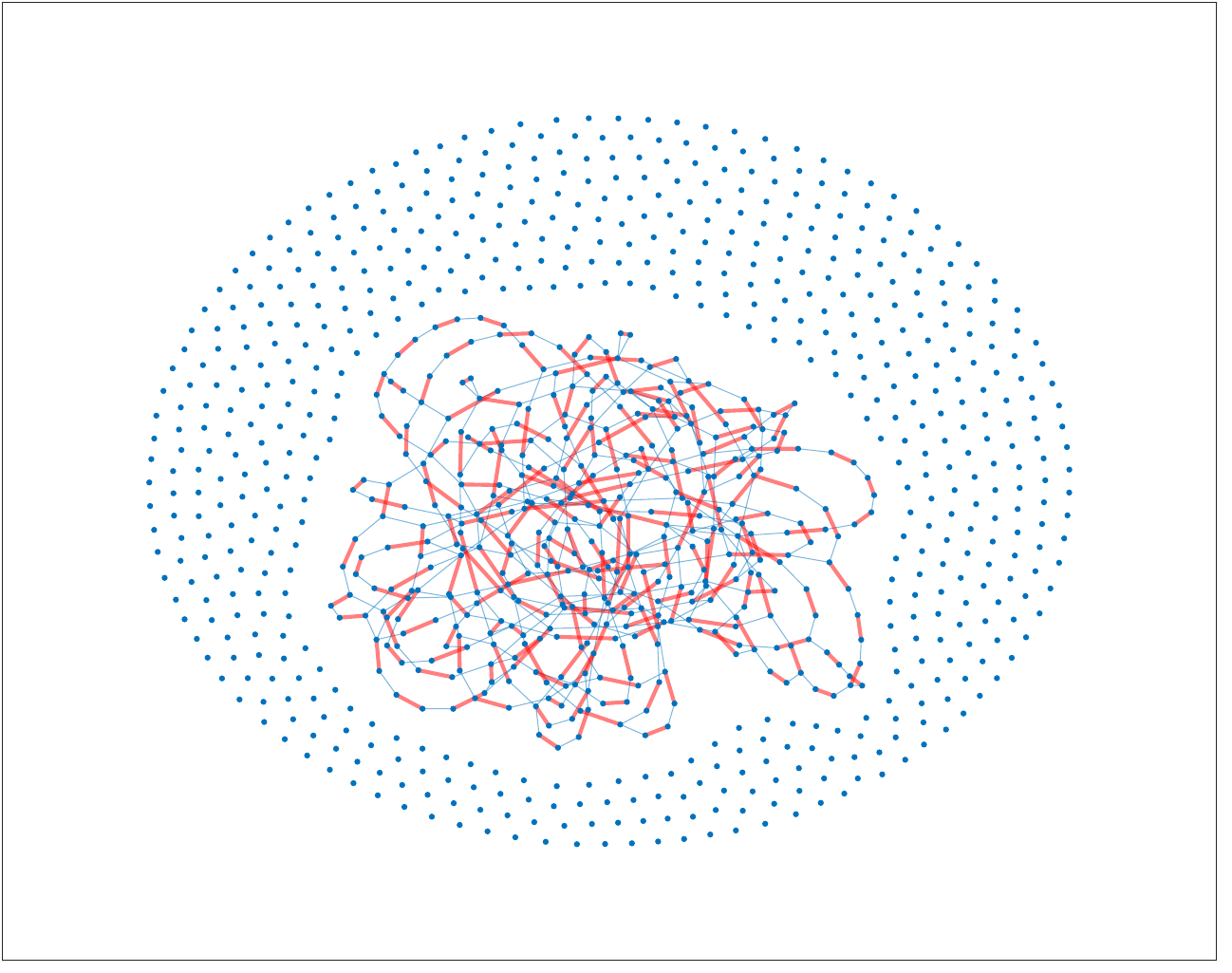}
\end{minipage}
\caption{The planted matching model ($k=1$ and $\lambda=1.5$), with planted edges marked in red and unplanted edges marked in blue. 
     Left panel: The observed graph $G$; Right panel: The remaining core after applying the iterative pruning algorithm.}
     \label{fig:pruning}
\end{figure}

Note that this algorithm inspects each edge at most twice (once to compute all the vertex degrees initially, then another time over the course of the pruning iterations). 
Hence its runtime is linear in the number of edges. We call the final remaining graph a ``core''\footnote{We caution the reader that this core is different from the standard notion of $k$-core, which is determined by iteratively removing vertices with degree less than $k$ and all their incident edges.}, denoted by $C_n$. If the core $C_n$ is empty, then we correctly identify all the planted edges and achieve exact recovery. If the core $C_n$ has $o(n)$ edges, then we correctly identify all but $o(n)$ planted edges and achieve almost exact recovery. See~\prettyref{fig:pruning} for a graphical illustration.

Next, we show that with high probability, if $\lambda=o(1),$ then the core is empty; if $k\lambda<1,$ then the core has $o(n)$ planted edges. This result will imply that the iterative pruning algorithm achieves both the exact and almost exact recovery thresholds. In fact, we will prove a stronger result, characterizing the asymptotic number of planted edges in the core $C_n$. 
\begin{theorem}[Iterative pruning algorithm] \label{thm:size}
If $k\lambda \equiv c$ for a constant c, then for any planted edge $e,$
\begin{align}
\lim_{n\to \infty} \prob{e \in C_n \mid H^*} = (1-\rho)^2,
\end{align}
where if $c\le 1$, then $\rho=1$; if $c>1$, then $\rho<1$ is the fixed point solution of 
\begin{equation}
\rho=\exp\left(-\lambda \left(1- \rho^k \right) \right) \label{eq:rho}   
\end{equation}
If $\lambda=o(1)$, then with high probability $C_n$ is empty. 
\end{theorem}
\begin{remark}
Note that according to the iterative pruning procedure,  $\hat{H} \subset H^*$. 
Thus~\prettyref{thm:size} shows that 
$$
\expect{\ell \left(H^*, \hat{H} \right) \mid H^* } = \frac{1}{|H^*|} \sum_{e \in H^*} \prob{e \in C_n \mid H^*} \to \left(1-\rho\right)^2.
$$
Then when $k\lambda \le 1$, 
$\ell(H^*,\hat{H})$ converges to $0$ in expectation and probability. 
However, the iterative pruning procedure may not achieve the minimum reconstruction error when $\lambda k >1.$ In fact,  additional structures in the core could help identify the planted edges. 
\end{remark}

\section{Proof Overview}\label{sec:proof_ideas}
In this section, we present the main proof ideas. 

\subsection{Alternating Circuits}
For ease of visualization, we color the planted edges red and unplanted edges blue. Our starting point is the following key observation: for any $k$-factor $H$, the symmetric difference $H\Delta H^*$ forms an even graph with balanced red and blue degrees. Consequently, it can be decomposed into a union of disjoint \emph{alternating circuits} (cycles with possibly repeated edges whose edges alternate between red and blue) (see e.g.~\cite[Theorem 1]{kotzig1968moves} and \cite[Theorem 1]{pevzner1995dna}). See Figure~\ref{fig:alt.circuit} for an illustrative example of a difference graph and its corresponding alternating circuit for $k=2$.

 \begin{figure}[ht]
   \begin{center}
\begin{tikzpicture}
    \def\pentagonSize{1.5cm}
    \draw[thick,red] (0:\pentagonSize)
        \foreach \i in {0,72, 144,216,288} {
            -- (\i:\pentagonSize) 
        }
        -- cycle; 
    \draw[thick,blue] (0:\pentagonSize)
        \foreach \i in {0,144,288,72,216} {
            -- (\i:\pentagonSize)
        }
        -- cycle; 

    \foreach \i/\label in {0/1,72/2,144/3,216/4,288/5} {
        \node at (\i:1.2*\pentagonSize) {\label};
    }
\end{tikzpicture}
\end{center}
\caption{An example of a difference graph with 5 vertices and a corresponding alternating circuit. One possible traversal is \(1 \to 2 \to 5 \to 1 \to 4 \to 5 \to 3 \to 4 \to 2 \to 3 \to 1\), where the edges alternate between red and blue along the circuit.}
 \label{fig:alt.circuit}
\end{figure}

This decomposition allows us to enumerate $k$-factors according to their Hamming distance from $H^*$, as formalized in the following lemma. 
\begin{lemma}[Enumerating $k$-factors]\label{lmm:enumeration}
For any $H^* \in \calH$ and any $0 \le t \le kn/2,$ 
\begin{align}
\left|\{ H \in \calH: |H \Delta H^*| = 2t  \}\right| \le \binom{kn/2}{t} \frac{(2t)!}{2^t t!}
\le \binom{kn/2}{t} 2^t t!
\le (kn)^t e^{-t(t-1)/(kn)}. \label{eq:M_ell_bound}
\end{align}
\end{lemma}
\begin{proof}
It suffices to count all possible unions of alternating circuits with a total length $2t$. Observe that such a union corresponds to selecting
$t$ planted edges and forming a perfect matching between the $2t$ endpoints. 
Since there are only $kn/2$ planted edges,
the number of ways to select $t$ planted edges is $\binom{kn/2}{t}$. Next, the $2t$ endpoints of these edges can be paired in  $(2t-1)!! = (2t)!/ (2^t t!)$ ways. Consequently, we deduce the first inequality. 
The second inequality~\prettyref{eq:M_ell_bound} holds due to $(2t)!/(t!)^2 =\binom{2t}{t} \le 2^{2t}$ 
and the last inequality holds because 
$$
\binom{kn/2}{t} 2^t t! \le \prod_{i=0}^{t-1} \left(kn -2i\right) 
= (kn)^t \prod_{i=1}^{t-1} \left( 1- \frac{2i}{kn}\right)
\le (kn)^t   \exp\left( - \frac{t(t-1)}{kn} \right). 
$$
\end{proof}

Using~\prettyref{lmm:enumeration} and the first-moment method, we can then show that with high probability, for any $k$-factor $H$ in the observed graph, $\ell(H^*,H)=0$ when $\lambda k=o(1)$, $\ell(H^*, H) =o(1)$
when $\lambda k \le 1$, achieving exact and almost exact recovery, respectively (and thus proving the positive side of Theorem \ref{thm:exact_recovery}). Theorem \ref{thm:almost_perfect} follows from a similar argument, and the achievability of partial recovery
when $\lambda =O(1)$ (Theorem \ref{thm:partial}) follows from the simple observation that with high probability the background graph $G_0$ contains $\Theta(n)$ isolated nodes whose incident edges in $G$ must be planted.

The impossibility of exact recovery when $k\lambda=\Omega(1)$ follows by proving the existence of an alternating cycle of length $4$ in the observed graph $G$. In contrast, proving the impossibility of almost exact recovery is significantly more challenging and requires an in-depth analysis of the posterior distribution, as we explain below. 

\subsection{Proof Ideas for Theorem~\ref{thm:impossibility}}

 A  crucial observation is that while a random draw $\tilde{H}$ from the posterior distribution~\eqref{eq:posterior_distribution} may not minimize the reconstruction error, its error is at most twice the minimum. Indeed, we can relate $\ell(H^*,\tilde{H})$ to $\ell(H^*,\hat{H})$ where $\hat{H}$ is any estimator, as follows: for any $D>0$
\begin{align}
\mathbb{P}\left(\ell(H^*,\tilde{H}) < 2D \right) &\ge \mathbb{P}\left(\ell(H^*,\hat{H})< D,\ell(\tilde{H},\hat{H}) <D \right) \nonumber \\
&=\expects{\mathbb{P}\left(\ell(H^*,\hat{H}) < D \mid G \right)\cdot \mathbb{P}\left(\ell(\tilde{H},\hat{H}) < D \mid G\right)}{G} \nonumber \\
&= \expects{\left(\mathbb{P}\left(\ell(H^*,\hat{H}) < D \mid G\right)\right)^2}{G} \nonumber \\
& \ge \left(\mathbb{P}\left(\ell(H^*,\hat{H})< D \right)\right)^2 \label{eq:posterior-tail}, 
\end{align}
where the first and second equalities hold because $H^*$ and $\tilde{H}$ are two independent draws from the posterior distribution conditioned on $G$; the last inequality holds by Jensen's inequality.

 Therefore, it suffices to prove that the posterior sample $\tilde{H}$ has $\Omega(1)$ reconstruction error, which further reduces to demonstrating that the observed graph $G$ contains many more $k$-factors that are far from $H^*$ than those close to $H^*$. Using~\prettyref{lmm:enumeration}, a simple first-moment analysis bounds the number of $k$-factors in $G$ that are close to $H^*$. 
 
 To lower-bound the number of $k$-factors that are far away from $H^*$, recall that for any $k$-factor $H$, $H\Delta H^*$
can be decomposed into a union of disjoint alternating circuits. Moreover, given any union of disjoint alternating circuits $C$, the XOR $C\oplus H^*$ is a $k$-factor. Therefore, it suffices to show that when $\lambda k\geq 1+\epsilon$, there exist $e^{\Omega(n)}$ many alternating circuits of length $\Omega(n)$ with high probability. However, lower-bounding the number of long circuits is challenging; 
a naive second-moment analysis does not work due to the excessive correlations among these long alternating circuits.  

To overcome this challenge, 
we extend the two-stage constructive argument in~\cite{Ding2023} from the planted matching model with $k=1$ to $k \ge 2$: First, we construct $\Theta(n)$ many alternating paths of constant length via a carefully designed neighborhood exploration process; Second, we connect these alternating paths to form exponentially many
distinct alternating cycles using previously reserved edges via a sprinkling technique.

In more detail, we first reserve $\Theta(n)$ vertex-disjoint edges in $H^*$, which are chosen independently of $G_0$. Using the remainder of the graph, we then construct $\Theta(n)$ two-sided alternating trees. Each tree begins with a red edge, whose endpoints are the left and right roots. The trees are built by a breadth-first search, with odd layers added via blue edges and even layers added via red edges (see Figure \ref{fig:tree}). Crucially, we need to ensure that for any vertex $v$ added in an odd layer, none of its red neighbors have already been included in the tree (or any previously constructed trees). We therefore keep track of the \emph{full-branching} vertices- those which can safely be added at odd layers. Viewing the tree construction process as adding two layers at a time, each side of the tree is well-approximated by a branching process whose offspring distribution has a mean of $\lambda k$. If $\lambda > \frac{1}{k}$, then the branching process is supercritical, and hence has a (quantifiable) nonzero probability of survival. By making the comparison to a binomial branching process rigorous, we are able to lower-bound the probability that a given tree grows to a prescribed (constant) size on both sides.

\begin{figure}[t]
    \centering    \includegraphics[scale=0.3]{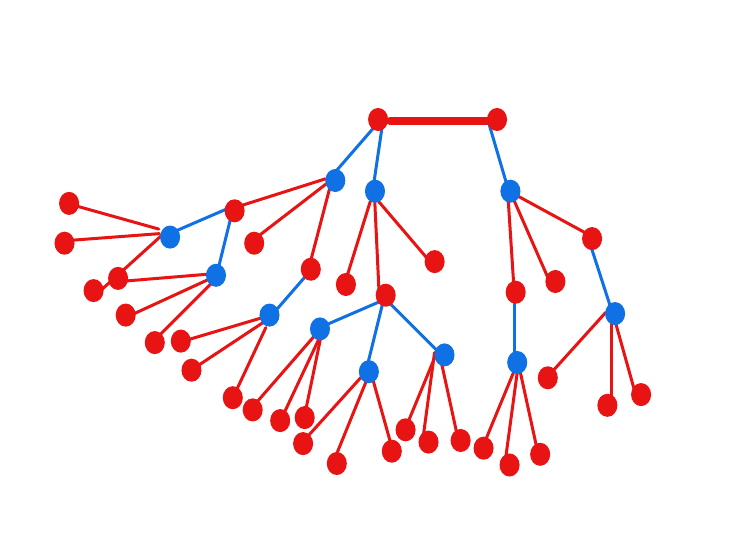}
    \caption{A two-sided alternating tree up to depth four, where the top (bolded) edge connects the two roots. Red and blue vertices are indicated. Here, all blue vertices have $k=3$ red descendants, as is the case in the tree construction.}
    \label{fig:tree}
\end{figure}

The tree construction process provides a linear number of large, two-sided alternating trees. Next, we assemble these trees into cycles. We first introduce some terminology: we say that a tree vertex is blue (resp. red) if the edge to its parent is blue (resp. red). The reserved edges are divided evenly into ``left'' and ``right'' sets $E_L^{\star}$ and $E_R^{\star}$. Additionally, for every reserved edge, one endpoint is designated as the ``tree-facing'' vertex while the other is designated as the ``linking'' vertex. The trees will be connected by a five-edge construction, as in Figure \ref{fig:five-edge-construction}. For the purposes of the construction, we say that a tree $L_i$ is \emph{blue-connected} to a tree-facing endpoint $u$ in $E_L^{\star}$ if some red vertex in $L_i$ is connected to $u$ by a blue edge. Similarly,  we say that a tree $R_j$ is blue-connected to a tree-facing endpoint $v$ in $E_R^{\star}$ if some red vertex in $R_j$ is connected to $v$ by a blue edge. If the corresponding linking vertices are connected by a blue edge (as in the long blue edge in Figure \ref{fig:five-edge-construction}), then we say that the $i^{\mathrm{th}}$ and  $j^{\mathrm{th}}$ trees are connected by a ``five edge construction'' (which is comprised of the three blue edges and two red edges which are bolded).

Observe that if there exists a sequence $i_1, \dots, i_m, i_{m+1} \equiv i_1$ where the $i_j^{\mathrm{th}}$ and  $i_{j+1}^{\mathrm{th}}$ trees are connected by a five edge construction for all $j \in [m]$, and no two of these five-edge constructions share a reserved edge, then we can form an alternating cycle that passes through these trees. We would like to argue that there are many long alternating cycles. To do so, we focus on trees $(L_i, R_i)$ that are blue-connected to many tree-facing endpoints, and discard the rest. Choosing a suitable constant $d$, we identify a set of trees that are blue-connected on both sides to $d$ tree-facing endpoints, associating each selected tree with $d$ tree-facing endpoints. Importantly, the sets of endpoints associated to different trees must be disjoint.

After discarding some trees, we are left with a collection of $m$ trees that are each red-connected to at least $d$ tree-facing vertices on each side. Thus, every pair $(L_i, R_j)$ is connected via a linking edge, and thus a five-edge construction, with probability at least $q :=1 - (1-p)^{d^2}$. Moreover, since we have associated each tree with a disjoint set of tree-facing neighbors, it follows that each pair $(L_i, R_j)$ is independently connected via a five-edge construction with probability $q$. We form an auxiliary bipartite graph with $m$ nodes on each side, with the left side corresponding to left trees and the right side corresponding to right trees. The bipartite graph contains a perfect red matching (to symbolize the root edges of the trees) and additionally contains blue edges independently with probability $q$. It follows that any cycle in the auxiliary graph induces a corresponding cycle which is at least as long in the original graph $G$. We thus apply known results from \cite{Ding2023} to lower-bound the number of long alternating cycles in two-colored bipartite graphs containing a perfect matching.

\begin{figure}[t]
    \centering
\includegraphics[scale=0.45]{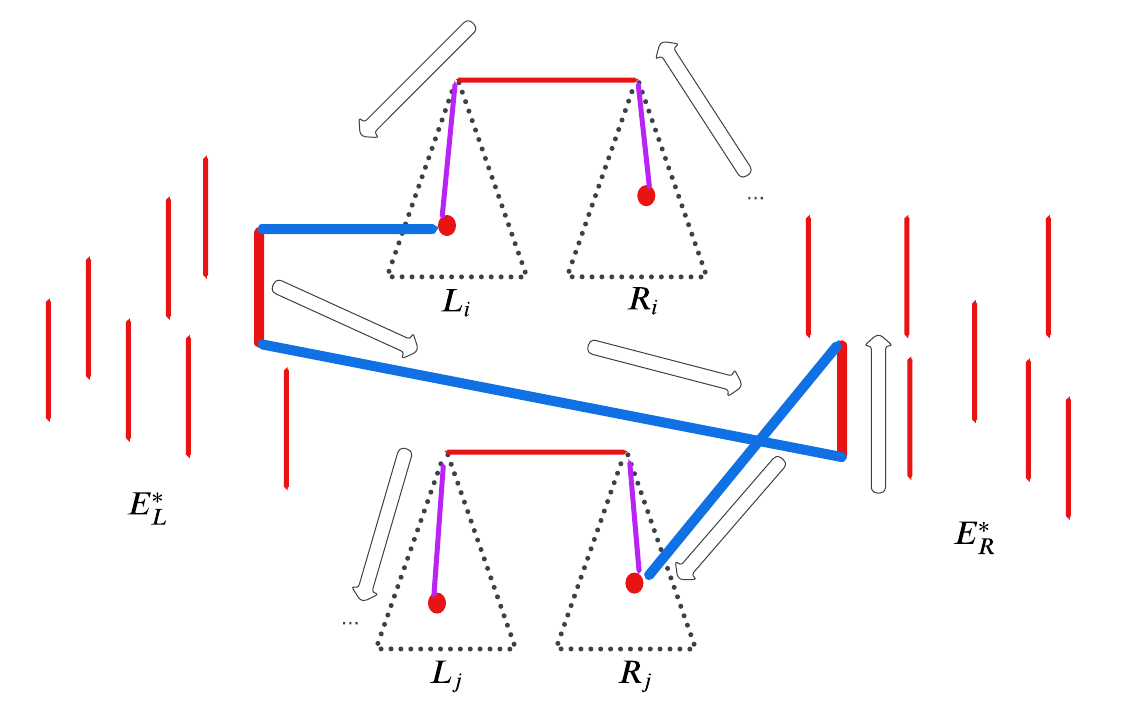}
    \caption{A portion of an alternating cycle, with the cycle order indicated by the arrows. Alternating paths are indicated by purple lines. Reserved edges in $E_L^{\star}$ and $E_R^{\star}$ are drawn vertically, with the top endpoints as the tree-facing endpoints, and the bottom endpoints as the linking endpoints. 
    Here, the tree $L_i$ is connected to the tree $R_j$ by a five-edge construction comprising three blue edges and two red edges (bolded): $L_i$ is blue-connected to the tree-facing endpoint of an edge $e_L$ in $E_L^{\star}$, and similarly $R_j$ is blue-connected to the tree-facing endpoint of an edge $e_R$ in $E_R^{\star}$. In turn, the linking endpoints of $e_L$ and $e_R$ are connected by a blue edge.}
    \label{fig:five-edge-construction}
\end{figure}

While the alternating cycle construction closely follows the steps in~\cite{Ding2023}, there are some differences. Notably, when $k \ge 2$, coupling the neighborhood exploration process with a branching process of $k\lambda$ mean offspring requires the vertices  to be ``full-branching'' with all their $k$ planted neighbors unvisited before. A more minor difference is that \cite{Ding2023} focuses on the bipartite graph, which simplifies the sprinkling step. In comparison,  our setting requires additional care to handle the unipartite nature of the graph; in particular, we manually assign tree roots and reserve edges to either the left or the right.

\FloatBarrier

\subsection{Proof Ideas for Theorem~\ref{thm:nothing}}
Establishing the impossibility of partial recovery when $\lambda=\omega(1)$ reduces to showing the posterior sample $\tilde{H}$ shares $o(n)$ edges with $H^*$, or equivalently, almost all $k$-factors in the observed graph $G$ are almost disjoint from $H^*$.
An upper bound on the number of $k$-factors in $G$ that share $\Omega(n)$ edges with $H^*$ can again be obtained using~\prettyref{lmm:enumeration} and the first-moment analysis. To lower-bound the number of $k$-factors that are almost disjoint from $H^*$, instead of applying an explicit constructive argument as aforementioned, we can simply bound the total number of $k$-factors in $G$ from below using the expected number of $k$-factors in $G_0.$ This follows from a simple yet elegant change-of-measure argument in~\cite[Lemma 3.9]{Mossel2023}, itself inspired by an earlier work on random constraint satisfaction problems \cite{achlioptas2008algorithmic}. 

\subsection{Proof Ideas for Theorem~\ref{thm:size} }
Finally, we turn to characterizing the size of the core which is left after running the iterative pruning algorithm. 
The key is to characterize the expected number of planted edges in the core, $C_n$, or equivalently, $\prob{e \in C_n \mid H^*}$ for a given planted edge $e$. 

 We first show that $\prob{e \in C_n \mid H^* } \le (1-\rho)^2+o(1).$ Here is the high-level idea. We construct the two-sided local neighborhood rooted at the planted edge $e$ of depth $d$, which consists of all alternating paths of length $d+1$ starting from $e.$ When $d=o(\log n)$, the local neighborhood at each side can be coupled with a Galton-Watson tree, where the new edges branching out at each layer alternate between $\Pois(\lambda)$ blue edges and fixed $k$ red edges. Let $\rho_d$ denote the probability that the Galton-Watson tree dies out in $d$ depth. Then both sides of the tree do not die out within depth $d$ with probability $(1-\rho_d)^2$. 
We then argue that if $e \in C_n$, both sides of the tree cannot die out for any finite $d$; otherwise, edge $e$ would be removed by the iterative pruning algorithm. Thus, $\prob{e \in C_n} \le (1-\rho_d)^2 + o(1) $ for any constant $d$. Finally, choosing $d$ to slowly grow with $n$, we show $\rho_d \to \rho$ and complete the proof of 
$\prob{e \in C_n \mid H^* } \le (1-\rho)^2+o(1).$

Next, we prove that $\prob{e \in C_n} \ge (1-\rho)^2-o(1)$. Note that this is trivially true when $k\lambda\le 1$ since $\rho=1$. Thus it suffices to focus on $k\lambda>1$. The high-level idea is as follows.  We first claim that if $e$ belongs to an alternating cycle, then $e \in C_n$. Thus, it suffices to lower-bound the probability that $e$ belongs to an alternating cycle. To do this, we first reserve a set of $\gamma n$ red edges. Then we build the two-sided tree rooted at $e$ as we did in the impossibility proof of almost exact recovery. 
Finally, we create an alternating cycle by connecting both sides of the tree to the same reserved red edge. With probability approximately $(1-\rho)^2,$ both sides of the tree do not die out. When this happens, we can grow the tree until both sides contain $\sqrt{n \log n}$ leaf vertices. Then with high probability, there exists a reserved red edge whose endpoints are connected to the two sides of the tree via blue edges, forming an alternating cycle. Hence, the probability that $e$ belongs to an alternating cycle is at least $(1-\rho)^2-o(1)$. Interestingly, this cycle construction differs from the one used in the impossibility proof for almost exact recovery. Here, we build a single large two-sided tree of size $\sqrt{n\log n}$ and connect the two sides via a three-edge (blue-red-blue) sprinkling. In comparison, the impossibility proof for almost exact recovery builds $\Theta(n)$ many small two-sided trees of size $\Theta(1)$, which are then connected using a five-edge (blue-red-blue-red-blue) sprinkling.

To prove exact recovery when $\lambda = o(1),$ we  extend the notion of an alternating cycle to an ``almost'' alternating cycle, where edges alternate in color except at the transition between the last and first edges.  We then show that if graph $G$ contains no such ``almost'' alternating cycle, the core $C_n$ must be empty. Finally, we establish that if $\lambda  =o(1)$, then with high probability, the graph does not contain any ``almost'' alternating cycle and hence the core $C_n$ is empty.

\section{Proofs for exact recovery}\label{app:exact}
\begin{proof}[Proof of~\prettyref{thm:exact_recovery}] (Positive direction). Suppose $\lambda=o(1)$. Then 
\begin{align}
\prob{\bigcup_{H \in \mathcal{H}, H \neq H^*} \{ H \subseteq G\}  \mid H^*}
& \leq  \sum_{t \ge 2} \sum_{H \in \calH: \left|H \triangle H^*\right|=2t }
\prob{H \subset G \mid H^*}
\nonumber \\ 
& \overset{(a)}{\leq}  \sum_{t\geq 2}(kn)^t \left(\frac{\lambda}{n}\right)^t
= \frac{(k\lambda)^2}{1-k\lambda}=o(1), 
\label{eq:exact_positve}
\end{align}
where the inequality $(a)$ follows from~\prettyref{eq:M_ell_bound} in~\prettyref{lmm:enumeration}.

(Negative direction). First, observe that for every edge in $H^*$, there are at most $2k$ vertices that are within distance $1$ from either of its endpoints in $H^*$. Consequently, there are at most $2k^2$ planted edges with one of these vertices as an endpoint. Thus, there must exist a collection of $\lceil \frac{kn}{4k^2}\rceil$ planted edges such that (1) no two of these edges share a vertex; and (2) no two of these edges have a planted edge between their vertices. Let $(v_1,v_1'),(v_2,v_2'),...,(v_{\lceil \frac{n}{4k}\rceil},v_{\lceil \frac{n}{4k}\rceil}')$ be any such collection of planted edges. 
Given any $1\le i<j\le \lceil \frac{n}{4k}\rceil$, let $\calE_{ij}$ denote the event that the graph $G$ has an edge between $v_i$ and $v_j$ and an edge between $v'_i$ and $v'_j$. If $\calE_{ij}$ holds, then $(v_i,v_i',v'_j,v_j)$ forms an alternating cycle of length $4.$ Replacing the edges between $(v_i,v'_i)$ and $(v_j,v'_j)$ in $H^*$  with the edges between $(v_i,v_j)$ and $(v'_i,v'_j)$ yields another $k$-factor $H \neq H^*$ contained in the graph $G$. Therefore, it remains to prove $\prob{\cup_{i<j}\calE_{ij}}=\Omega(1).$
By construction, $\prob{\calE_{ij}}=\lambda^2/n^2$.  Furthermore, $\{(v_i,v_j),(v'_i,v'_j)\}$ is disjoint from $\{(v_{i'},v_{j'}),(v'_{i'},v'_{j'})\}$ for all $(i',j')\ne (i,j)$, so the events $\calE_{ij}$ are mutually independent for all $i<j$. Hence,  
$$
\prob{\cup_{i<j} \calE_{ij}}
= 1- \prob{\cap_{i<j} \calE^c_{ij}}
= 1- \prod_{i<j} \prob{\calE^c_{ij}}
= 1-(1-\lambda^2/n^2)^{\lceil \frac{n}{4k}\rceil(\lceil \frac{n}{4k}\rceil-1)/2}=\Omega(1),
$$
where the last equality holds by the assumption that $\lambda =\Omega(1)$
and $k$ is a fixed constant. 
\end{proof}

\section{Proofs for almost exact recovery}

We first present the proof for the positive direction of almost exact recovery.

\begin{proof}[proof of~\prettyref{thm:almost_perfect}] We apply the first-moment method following~\cite{Ding2023}. In particular, 
\begin{align*}
\mathbb{P}\left\{\ell\left(\widehat{H},H^*\right)\geq 2\beta  \mid H^* \right\}
& \leq  \mathbb{P}\left\{\exists H \in \calH, \left|H \triangle H^*\right| \ge \beta kn: H \subset G \mid H^* \right\}\\
& \leq   \sum_{t\geq \beta k n/2}^{kn/2}
\sum_{H \in \calH: \left|H \triangle H^*\right|=2t}
\mathbb{P}\left\{ H \subset G \mid H^* \right\}\\
& \stackrel{(a)}\leq  \sum_{t\geq \beta k n/2}^{kn/2}(kn)^t e^{-t(t-1)/(kn)}\left(\frac{\lambda}{n}\right)^t\\
& \leq  e^{1/2}\sum_{t\geq \beta k n/2}
\left( (1+\epsilon) e^{-\beta/2} \right)^{t} \leq  e^{1/2} \frac{e^{-\beta^2 kn/8}}{1-e^{-\beta/4}},
\end{align*}
where step (a) follows from~\prettyref{eq:M_ell_bound} in~\prettyref{lmm:enumeration},
and the last inequality holds for all $\beta \ge 4\log (1+\epsilon)$, since 
\[(1+\epsilon) e^{-\beta/2} = (1+\epsilon) e^{-\beta/4} \cdot e^{-\beta/4} \leq  (1+\epsilon) e^{-4 \log(1+\epsilon)/4} \cdot e^{-\beta/4} =  e^{-\beta/4}.\]
Therefore,
\begin{align*}
&\expect{\ell\left(\widehat{H},H^*\right) \mid H^* }\\
& \leq\expect{\ell\left(\widehat{H},H^*\right) \indc{\ell\left(\widehat{M},M^*\right) < 2\beta } \mid H^*} + \expect{\ell\left(\widehat{H},H^*\right) \indc{\ell\left(\widehat{H},H^*\right) \ge 2\beta } \mid H^*} \\
& \le 2 \beta + 2 \prob{\ell\left(\widehat{H},H^*\right) \ge 2\beta \mid H^*} \le 2\beta +  2 e^{1/2}\frac{e^{-\beta^2 kn/8}}{1-e^{-\beta/4}}.
\end{align*}
By choosing $\beta$ according to~\eqref{eq:beta_def}
so that $e^{-\beta^2 k n/8} \le \beta^2 /8$ and $1-e^{-\beta/4} \ge \beta/8$ 
we have
\begin{align*}
\expect{\ell\left(\widehat{M},M^*\right) \mid H^*} \leq  2\beta +  2 e^{1/2}\frac{\beta^2/8}{\beta/8} 
\le 6 \beta. 
\end{align*}
\end{proof}

Next, we present the  proof of \prettyref{thm:impossibility}, the negative direction of almost exact recovery. In view of~\prettyref{eq:posterior-tail}, it suffices to consider the estimator $\tilde H$ which is sampled from the posterior distribution $\mu_G$. 
In order to get a probabilistic lower-bound on $\ell(H^*,\tilde{H})$, we define the sets of good and bad solutions respectively as
\begin{align*}
\Mgood (G) = & ~ \left\{ H \in \calH: \ell( H, H^*) < \frac{2 \delta}{k}, H \subset G \right\} \\
\Mbad (G) = & ~ \left\{ H \in \calH: \ell(H,H^*)  \geq \frac{2 \delta}{k}, H \subset G \right\}.
\end{align*}
The value $\frac{2\delta}{k}$ is chosen since $\ell(H, H^*) = \frac{2\delta}{k}$ means that $|H^* \Delta H| = \frac{2\delta}{k} \cdot \frac{nk}{2} = \delta n$ (using \eqref{eq:risk}).
Recall that the posterior distribution $\mu_G$ is the uniform distribution over all possible $k$-factors contained in the observed graph $G.$ Therefore, by the definition of $\tilde H$, we have 
\begin{align}
\prob{ \ell(\tilde H,H^*) <  \frac{2 \delta}{k}  \mid G, H^*} = \frac{|\Mgood |}{ |\Mgood| + |\Mbad|} . \label{eq:posterior_error_bound}
\end{align}

Next, we bound $\left| \Mgood \right|$ and $\left| \Mbad \right|$.
\begin{lemma}
\label{lmm:Mgood}	
Assume that~\prettyref{eq:impossibility} holds for some arbitrary constant $\epsilon>0$. 
Then for any $\delta>0$, with probability at least $1-(k\lambda)^{  - \frac{\delta n}{2}  }$, 
	\begin{equation}
	\left| \Mgood \right|\leq  \frac{k\lambda}{k\lambda -1} (k\lambda)^{\delta n}.
\label{eq:Mgood}
	\end{equation}
conditioned on any realization of $H^*$.
\end{lemma}

\begin{lemma}
\label{lmm:Mbad}
Suppose \prettyref{eq:impossibility} holds for some arbitrary constant $\epsilon>0$. 
There exist constants $c_0$ and $c_1$ that only depend on $\epsilon, k$, 
such that for all $\delta \leq c_0$, 
with probability at least $1-e^{-\Omega(n)}$, 
		\begin{equation}
	\left| \Mbad  \right| \geq e^{c_1 n}.
	\label{eq:Mbad}
	\end{equation}
 conditioned on any realization of $H^*$.
\end{lemma}

\begin{proof}[Proof of~\prettyref{thm:impossibility}]
Observe that we can assume $\lambda k = 1+\epsilon$ without loss of generality; any estimator that works for $\lambda k>1+\epsilon$ can be converted to an estimator for $\lambda k=1+\epsilon$ by adding extra edges to the graph before computing the estimator. 

Given the above two lemmas, Theorem \ref{thm:impossibility} readily follows. Indeed, combining
\prettyref{lmm:Mgood} and \prettyref{lmm:Mbad} and choosing 
$\delta =\min\{c_0, c_1/(2\log(k\lambda))\}$ so that $(k\lambda)^{\delta n} \leq e^{c_1 n/2} $, we obtain
 $$
 \frac{|\Mgood |}{ |\Mgood| + |\Mbad| } \leq
 \frac{|\Mgood |}{  |\Mbad| }
 \le \frac{k\lambda}{k\lambda-1}\cdot (k\lambda)^{\delta n} e^{-c_1n} 
 \leq \frac{k\lambda}{k\lambda-1} e^{-c_1n/2},
 $$ 
 with probability $1- e^{-\Omega(n)}$. It then follows from \prettyref{eq:posterior-tail} and \prettyref{eq:posterior_error_bound} that for any estimator $\hat{H}$, 
 \begin{align*}
    \left[\prob{\ell(H^*, \hat{H}) < \frac{\delta}{k} } \right]^2 &\leq \prob{\ell(H^*, \tilde{H}) < \frac{2\delta}{k} }\\
    &= \mathbb{E}_{G, H^*} \left[ \prob{\ell(H^*, \tilde{H}) < \frac{2\delta}{k} \mid G, H^* }\right]\\
    &= \mathbb{E}_{G, H^*} \left[\frac{|\Mgood |}{ |\Mgood| + |\Mbad| }   \right]\\
    &\leq \left(1 - e^{-\Omega(n)}\right) \frac{k\lambda}{k\lambda-1} e^{-c_1n/2} + e^{-\Omega(n)} = e^{-\Omega(n) }.
 \end{align*}

 Finally, we have
$$
\expect{ \ell(H^*, \hat{H})} 
\ge \prob{ \ell(H^*, \hat{H}) \ge \frac{\delta}{k}} \cdot \frac{\delta}{k}
\ge \left( 1- e^{-\Omega(n)} \right) 
\frac{\delta}{k} \ge \frac{\delta}{2k}.
$$
Taking $\epsilon' = \frac{\delta}{k}$ completes the proof.
\end{proof}
\begin{remark}
The above proof also shows that with probability at least $1- e^{-\Omega(n)}$, at least  $(1- \frac{k\lambda}{k\lambda-1} e^{-c_1n/2})$ fraction of $k$-factors in graph $G$ satisfy 
$\ell(H,H^*) \ge \frac{2\delta}{k}$. 
\end{remark}
\begin{proof}[Proof of~\prettyref{lmm:Mgood}] 
Applying~\prettyref{lmm:enumeration}, we get that 
\begin{align*}
\mathbb{E}[\left|\Mgood\right| \mid  H^*] &= \sum_{t < \delta n} \left|\{H \in \mathcal{H} : |H \Delta H^*| = t\} \right| \cdot \left(\frac{\lambda}{n} \right)^{t/2}\\
&\leq \sum_{t < \delta n} (kn)^{t/2} \cdot \left(\frac{\lambda}{n} \right)^{t/2}= \sum_{t < \frac{\delta n}{2}} (k\lambda)^{t}\\
&\leq \frac{(k\lambda)^{\frac{\delta n}{2}+1}-1}{k\lambda -1} \leq \frac{k\lambda}{k\lambda -1} (k\lambda)^{\frac{\delta n}{2}}.
\end{align*}
The conclusion follows by applying Markov's inequality.
\end{proof}

In order to prove \prettyref{lmm:Mbad}, we will provide an algorithm for constructing a large number of $k$-factors in $|\Mbad|$. The initialization step, defined in Algorithm \ref{alg:matching}, reserves a set of vertex-disjoint edges from the graph $H^*$. These reserved edges will be used to connect the trees we will find into long cycles in the second stage.  Note that the algorithm (and the others that follow) require knowledge of $H^*$, so this is meant as a construction by the analyst rather than a procedure of the estimator. 

\begin{algorithm}[b]
\caption{Reserve Edges}\label{alg:matching}
\begin{algorithmic}[1]
\Statex{\bfseries Input:}Graph $H^*$ on $n$ vertices, $m \in \mathbb{N}$
\Statex{\bfseries Output:}A set $E^\star$ of $m$ vertex-disjoint red edges of $G$, available vertices $\mathcal{A}$, and full-branching vertices $\mathcal{F}$
\State Let $E^{\star} = \emptyset$, $S = E(H^*)$.
\For{$i \in \{1, 2, \dots, m\}$}
        \State Choose an arbitrary edge $e$ from $S$ and add it to $E^{\star} $.
        \State Remove $e$ and all edges incident to either endpoint of $e$ from $S$.
\EndFor
\State Let $V_1$ be the set of endpoints of edges in $E^{\star}$. Let $V_2$ be the set of vertices  adjacent to a vertex in $V_1$ via a red edge.
\State Let $\mathcal{A} = [n] \setminus V_1$ be the set of \emph{available} vertices. Let $\mathcal{F} = [n] \setminus (V_1 \cup V_2)$ be the set of \emph{full-branching} vertices.
\end{algorithmic}
\end{algorithm}
Let us explain the definitions of $\mathcal{A}$ and $\mathcal{F}$ in the last step. These two sets will be updated in the tree-construction stage. The set $\mathcal{A}$ will remain the set of unreserved vertices that have not appeared as a vertex in any tree, while $\mathcal{F}$ will remain the set of unreserved vertices whose incident edges have not been inspected in the construction. Crucially, our initialization ensures that \emph{each vertex in $\mathcal{F}$ has exactly $k$ planted neighbors in $\mathcal{A}$} and hence the name of ``full-branching''. This fact will remain true throughout our construction. 

Algorithm \ref{alg:tree-construction} constructs two-sided alternating trees, according to the following definition; see also Figure \ref{fig:tree}.
\begin{definition}
A two-sided alternating tree, denoted $(L,R)$, contains a red edge connecting the roots of $L$ and $R$. The subtrees $L$ and $R$ alternate blue edges and red edges on all paths from the roots to the leaves. We also say that a vertex is blue (resp. red) if the edge from it to its parent is blue (resp. red).
\end{definition}

The algorithm constructs trees via a breadth-first exploration. As such, a \emph{queue} data structure is employed to ensure the correct visitation order. Generically, a queue is a collection of objects that can be added to (via the $\texttt{push}$ operation) or removed from (via the $\texttt{pop}$ operation). A queue obeys the ``first in first out'' rule with respect to adding and removing.

\begin{algorithm}[t]
\caption{Tree Construction} \label{alg:tree-construction}
\begin{algorithmic}[1]
\Statex{\bfseries Input:} Graph $G$ and $k$-factor $H^*$ on $n$ vertices, 
available vertices $\mathcal{A}$, full-branching vertices $\mathcal{F}$, tree count parameter $\gamma$, size parameter $\ell \in \mathbb{N}$.
\Statex{\bfseries Output:} A set $\mathcal{T}$ of two-sided trees where each side has at least $2\ell$ vertices.
\State Set $\mathcal{T} = \emptyset$. 
\For{$t \in \left\{1, 2, \dots, K := \frac{\gamma n}{2(2\ell+k)k}  \right\}$}
\State Select an arbitrary planted edge $(u_0,u_0')$ where $u_0,u_0'\in \mathcal{A}$. If no such edge exists, return $\texttt{FAIL}$. 
\State Initialize $T$ to be a two-sided tree containing only the center edge $(u_0,u_0')$. 
\State Remove $u_0$ and $u_0'$ from both $\mathcal{A}$ and $\mathcal{F}$. Remove all the planted neighbors of $u_0$ and $u_0'$ from $\mathcal{F}$. 
\State (Grow the left tree rooted at $u_0$.) Initialize the leaf queue to be $\mathcal{L} \gets \{u_0\}$, and the cumulative size to be $s \gets 1$. 
\label{step.root.init}
\While{$\mathcal{L} \neq \emptyset$ and $s < 2\ell$}
\State Let $u \gets \mathcal{L}.\texttt{pop}$. 
\State (Find the children of $u$.) Let $\mathcal{C}_u \gets \{v \in \mathcal{F} : (u,v) \text{ is an unplanted edge}\}$; i.e., $\mathcal{C}_u$ is the set of all full-branching unplanted neighbors of $u$. 
\For{$v \in \mathcal{C}_u$}
\If{$s\geq 2\ell$}
\State go to line \ref{step.right.tree}.
\EndIf
\If{$v \in \mathcal{F}$}
\State Attach $v$ to $T$ as a child of $u$. \label{step:check-full-branching}
\State Let $\mathcal{C}_v$ denote the planted neighbors of $v$.
\State Attach $\mathcal{C}_v$ to $T$ as children of $v$ (grandchildren of $u$). 
\State Set $s \gets s + k+1$, and update $\mathcal{L}$ as $\mathcal{L}.\texttt{push}(\mathcal{C}_v)$ 
\State\label{step:remove}Remove $v$ and $\mathcal{C}_v$ from $\calF$ and $\calA$. Remove all planted neighbors of $\mathcal{C}_v$ from $\calF$. 
\EndIf
\EndFor
\EndWhile\label{step.finish.left.growth}
\label{step.right.tree} \If{$s \geq 2\ell$}
\State Grow the right tree rooted at $u_0'$, analogously to lines \ref{step.root.init}-\ref{step.finish.left.growth}, initializing the leaf queue to be $\mathcal{L} \gets \{u_0'\}$. 
\State If the right tree also reaches a size of at least $2\ell$, then set $\mathcal{T} \gets \mathcal{T} \cup \{T\}$. 
\EndIf
\EndFor
\State Return $\mathcal{T}$. 
\end{algorithmic}
\end{algorithm}

Our goal is to connect the trees into cycles. To aid our analysis, the trees will be connected by a five-edge construction, as in Figure \ref{fig:five-edge-construction}. For the purposes of the construction, we say that a tree $L_i$ is \emph{blue-connected} to a tree-facing endpoint $v$ in $E_L^{\star}$ if some red vertex in $L_i$ is connected to $v$ by a blue edge. Algorithm \ref{alg:cycle-construction} then constructs an auxiliary bipartite graph which, at a high level, keeps track of the trees that are connected by a five-edge construction. We will show that the bipartite graph is well-connected, and hence has many long alternating cycles, which in turn translate into many long alternating cycles in $G$. Crucially, the bipartite graph will need to have independent blue edges, which correspond to the blue edges which connect linking endpoints. To ensure the independence, we will need to avoid \emph{collisions}, where two trees are blue-connected to the same tree-facing endpoint. These collisions are avoided in Algorithm \ref{alg:cycle-construction} by considering the trees sequentially, and only forming blue connections to unused tree-facing endpoints (see Figure \ref{fig:tree-bookkeeping}).

\begin{algorithm}[t]
\caption{Cycle Construction} \label{alg:cycle-construction}
\begin{algorithmic}[1]
\Statex{\bfseries Input:} Graph $G$ on $n$ vertices (with red subgraph $H^*$), tree count parameter $\gamma$, tree size parameter $\ell \in \mathbb{N}$, degree parameter $d$
\Statex{\bfseries Output:} A set of alternating cycles $\mathcal{C}$ on $G$
\smallskip
\State Set $\gamma = \tfrac{\epsilon}{10(1+\epsilon)}$. Apply Algorithm \ref{alg:matching} to input $(H^*, 2\gamma n/k)$, obtaining the set of reserved edges $E^{\star}$, the set of available vertices $\mathcal{A}$, and the set of full-branching vertices $\mathcal{F}$. We assume $|E^*|$ is even.
\State Let $\mathcal{T} = (L_i, R_i)_{i=1}^{K_1}$ be the output of Algorithm \ref{alg:tree-construction} on input $(G, 
\mathcal{A}, \mathcal{F}, \gamma, \ell)$.
\State Randomly partition $E^{\star}$ into two equally-sized sets $(E^{\star}_L, E^{\star}_R)$. 
For each $(u,v) \in E^{\star}$ with $u < v$, designate $u$ as the ``tree-facing'' vertex and designate $v$ as the ``linking'' vertex.
\State Initialize an empty (bipartite) graph $\overline{G}$.
\For{$i \in [K_1]$}
\If{$L_i$ is blue-connected to at least $d$ unmarked tree-facing endpoints among $E_L^{\star}$ and the same is true for $R_i$ with respect to $E_R^{\star}$}\label{step:collision}
\State Let the first $d$ of these edges be denoted $\mathcal{E}(L_i) \subset E_L^{\star}$ and $\mathcal{E}(R_i) \subset E_R^{\star}$. 
\State Mark all edges among $\mathcal{E}(L_i) \cup \mathcal{E}(R_i)$.
\State Include $i$ as a vertex on both sides of $\overline{G}$, and connect them by a red edge.
\EndIf
\EndFor
\For{$i \in [K_1]$}
\For{$j \in [K_1]$}
\If{both $i$ and $j$ are vertices in $\overline{G}$, and some linking endpoint in $\mathcal{E}(L_i)$ is connected to some linking endpoint in $\mathcal{E}(R_j)$ by a blue edge}
Connect $i$ and $j$ by a blue edge in $\overline{G}$. \;
\EndIf
\EndFor
\EndFor
\State Find the set of alternating cycles in $\overline{G}$, and return the corresponding set of alternating cycles in $G$. 
\end{algorithmic}
\end{algorithm}

\FloatBarrier

\begin{figure}[t]
    \centering
    \includegraphics[scale=0.25]{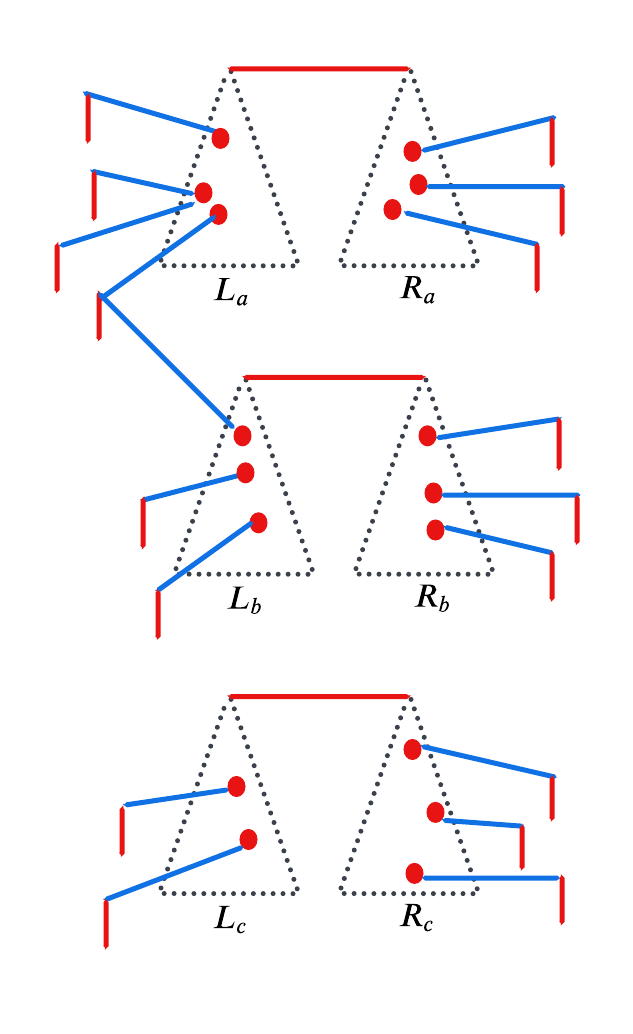}
    \caption{Trees are considered sequentially to avoid collisions, in alphabetical order. Here $d = 3$. The vertex $a$ is added to the bipartite graph $\overline{G}$ by~\prettyref{alg:cycle-construction}, while $b$ and $c$ are not. While $L_b$ and $R_b$ each connect to $3$ tree-facing vertices, one of the vertices connected to $L_b$ is already connected to $L_a$, forming a collision. The vertex $c$ is not included because $L_c$ is blue-connected to only two tree-facing vertices.}
    \label{fig:tree-bookkeeping}
\end{figure}

\subsection{Proof of~\prettyref{lmm:Mbad}}

\subsubsection{Tree construction}
Our first goal is to characterize the tree construction, ensuring that Algorithm \ref{alg:tree-construction} produces sufficiently many trees. 
\begin{proposition}\label{prop:tree.facts}
    The tree construction process ensures that: 
    \begin{enumerate}[(a)]
        \item Each two-sided tree contains at most $4\ell+2k$ vertices, with $2\ell+k$ on each side. 
        \item For each two-sided tree $T_k = (L_k, R_k)$ for which both sides contain at least $2\ell$ vertices, both sides contain at least $\ell$ red vertices. 
        \item Throughout the construction, the number of full-branching vertices satisfies $|\mathcal{F}|\geq n-5\gamma n$. 
    \end{enumerate}
\end{proposition}

\begin{proof}
    To prove (a), note that the left or right subtree construction is deemed complete when it contains at least $2\ell$ vertices, and the completion condition is checked every time we add a child vertex along with its $k$ planted neighbors, which implies that each side of each subtree has at most $2\ell+k$ vertices.

    Next we prove (b). Consider a two-sided tree $T_k$ whose left tree contains at least $2\ell$ vertices. Note that by construction, the number of vertices on the even layers is exactly $k$ times the number of vertices on the layer above, each of which has $k$ children. Therefore within the left subtree, the number of vertices on the even layers is at least $(2\ell-1)\cdot k/(k+1)+1\geq \ell$, where the $-1$ and $+1$ account for the root node. The same argument applies to the right subtree.
    
    To prove (c), recall that in the initialization step, $\mathcal{F} = [n] \setminus (V_1 \cup V_2)$, where $V_1$ is the set of vertices represented by the reserved (red) edges $E^{\star}$, and $V_2$ is the set of vertices adjacent to a vertex in $V_1$ by a red edge. Since Algorithm \ref{alg:cycle-construction} sets $|E^{\star}| = \frac{2\gamma n}{k}$, it follows that at the initialization step, 
    \[|\mathcal{F}| \geq n - |V_1| - |V_2| \geq n - k|V_1| = n - 4\gamma n.\]
    In the tree construction stage, $K=\frac{\gamma n}{2(2\ell+k)k}$ two-sided trees are constructed. By \prettyref{prop:tree.facts} (a), each tree contains at most $2(2\ell + k)$ vertices, totalling at most $\gamma n/k$ vertices in all trees. Furthermore, each vertex that is removed from $\mathcal{F}$ in the tree construction stage is either in a tree, or is a planted neighbor of a vertex in a tree. Since each vertex in a tree has at least one of its planted neighbors in the tree, we have at most $\frac{\gamma n}{k}+ \frac{\gamma n}{k} \cdot (k-1)=\gamma n$ vertices removed from $\mathcal{F}$ during the tree exploration process. Therefore the size of $\mathcal{F}$ remains above $n-4\gamma n - \gamma n = n - 5\gamma n$.

\end{proof}

For the remainder of this subsection, we will condition on the realization of $H^*$.
In order to characterize the size of the trees, we compare the trees to branching processes, where the offspring distribution is $k$ independent copies of a suitable binomial random variable. At a high level, the probability that a given tree reaches a prescribed depth can be related to the survival probability of the branching process. We need the following auxiliary result about the survival of a supercritical branching process.

\begin{lemma}\label{lmm:branching}
Suppose a branching process has offspring distribution with expected value $\mu$ and variance $\sigma^2$ for some $\mu>1$, we have
\begin{equation}
\mathbb{P}\{\text{Branching process survives}\} \geq \frac{\mu^2-\mu}{\mu^2-\mu+\sigma^2}. \label{eq:branchin-brocess-survival}
\end{equation}
\end{lemma}

We can now prove that sufficiently many large two-sided trees are constructed.
\begin{theorem}\label{theorem:enough-trees}
Suppose $k\lambda \ge 1+\epsilon$, and $k,\ell$ are constants. Recall the algorithm parameters $\gamma = \frac{\epsilon}{10(1+\epsilon)}$ (Algorithm \ref{alg:cycle-construction}) and $K = \frac{\gamma n}{2(2\ell +k)k}$ (Algorithm \ref{alg:tree-construction}).
Conditioned on any realization of $H^*$, with probability $1-O(n^{-3})$, Algorithm \ref{alg:tree-construction} yields at least $K_1$ two-sided trees $T_k = (L_k, R_k)$ for which both $L_k$ and $R_k$ contain at least $\ell$ red vertices, where 
\[
K_1 = \frac{K\epsilon^2}{2(\epsilon+2k)^2}.
\]
\end{theorem}

\begin{proof}
Note that by construction, each vertex on an odd layer of a tree has exactly $k$ children. Therefore the only source of randomness in the number of red vertices in $L_k$ and $R_k$ comes from when vertices attach to their parents via unplanted edges, i.e. when the tree grows to an odd layer. 

At a high level, we will compare each tree's growth to a (two-sided) branching process, lower-bounding the probability that a tree is grown successfully by the probability that the branching process survives. A challenge arises due to Step \ref{step:remove}, where we remove the red neighbors $\mathcal{C}_v$ from $\mathcal{F}$, where $\mathcal{C}_v$ is the set of red neighbors of a tree vertex $v$. The purpose of this removal is to maintain the invariant that any vertex in $\mathcal{F}$ has all of its red neighbors in the set $\mathcal{A}$. Still, we can control the number of vertices that are removed while the tree is still smaller than the target size, enabling a comparison to an auxiliary branching process.

Formalizing the comparison, construct $2K$ independent branching processes with offspring distribution $k\cdot \Binom(n-5\gamma n, \lambda/n)$, denoted by $B_1, \ldots, B_{2K}$. Let $p$ be the survival probability of $B_1$. To lower-bound $p$, we apply \prettyref{lmm:branching}
with
\begin{align*}
\mu &= k(n-5\gamma n) \cdot \frac{\lambda}{n},\\
\sigma^2 &= k^2(n-5\gamma n) \cdot \frac{\lambda}{n} \left(1-\frac{\lambda}{n}\right) \leq k\mu.
\end{align*}
We have the survival probability
\[
p\geq \frac{\mu^2-\mu}{\mu^2-\mu + \sigma^2}
\geq \frac{\mu-1}{\mu-1 +k}.
\]
Since $k\lambda\ge 1+\epsilon$, we have $\mu\geq (1+\epsilon)(1-5\gamma)= 1+\epsilon/2$. 
It follows that
\[
p\geq \frac{\epsilon}{\epsilon+2k}.
\]

We will construct a coupling such that 
for every $1 \le i\le 2K,$ as long as the $i^{\text{th}}$ tree (could be left side or right side) has not reached the size of $2\ell$, it has at least as many offspring as $B_i$ at each layer. Specifically, when the $i^{\text{th}}$ tree grows to an odd layer from a given parent node $u$, we sequentially check a full-branching node $v$ from the set $\calF$, reveal whether $v$ is connected to $u$ via a blue edge, and update the set $\calF$ accordingly. Crucially, the blue edge between $u$ and $v$ is distributed as $\Bern(\lambda/n)$, independently of everything else. Thus we can couple the blue edge between $u$ and $v$ with a new offspring in the branching process $B_i$ as follows. If the blue edge between $u$ and $v$ exists, we add a new offspring to $B_i$. Since the number of full-branching vertices satisfies $|\calF| \ge n-5\gamma n$ throughout the entire construction, there are only two possibilities. Either we check $n-5\gamma n$ full-branching nodes $v$, in which case we stop adding new offspring to $B_i$. Otherwise, the $i^{\text{th}}$ tree has reached the size of $2\ell$ nodes and the construction of the $i$-th tree is finished. In this case, we randomly add additional offspring to $B_i$  to ensure the offspring distribution of $B_i$ is exactly $\Binom(n-5\gamma n, \lambda/n)$. We can check that under this coupling, when the $i$-th tree has not reached the size of $2\ell $ nodes, it has at least as many offspring as $B_i$ at each layer. Therefore, 
$$
\mathds{1}\{ B_i \text{ survives}\} 
\le \mathds{1}\{ T_i \text{ contain at least } 2\ell \text{ vertices}\}
$$
In other words, when $T_i$ contains fewer than $2\ell$ nodes, $B_i$ must die out. It follows that 
    \begin{align*}
    & \prob{\sum_{i\leq K}\mathds{1}\{\text{Both sides of }T_k\text{ contain at least }2\ell \text{ vertices}\} < Kp^2/2}\\
    &\leq  \prob{\sum_{i\leq K} \mathds{1}\{ B_{2i-1} \text{ and } B_{2i} \text{ survive}\} <Kp^2/2}\\
    &\leq  \prob{\Binom(K,p^2)<Kp^2/2}\\
    &=  e^{-\Omega(n)}
\end{align*}
since $K=\Omega(n)$.

From \prettyref{prop:tree.facts} (b), if both sides of $T_k$ contain at least $2\ell$ vertices, we must have $|L_k|\geq \ell$ and $|R_l|\geq \ell$. We have shown that with probability $1-O(n^{-3})$, the number of trees satisfying $|L_k|\geq \ell$ and $|R_k|\geq \ell$ is at least
\[
\frac{Kp^2}{2}\geq \frac{K\epsilon^2}{2(\epsilon+2k)^2}=K_1.
\]
conditioned on the realization of $H^*$.
\end{proof}

Finally, we prove Lemma \ref{lmm:branching}.
\begin{proof}[Proof of Lemma \ref{lmm:branching}]
The proof mostly follows the derivations in~\cite[Chapter 2.1]{durrett2007random}.
Let $Z_m$ denote the number of vertices in generation $m$. Given $Z_{m-1}$, the conditional first and second moments of $Z_m$ satisfy
\begin{align*}
    \mathbb{E}[Z_m | Z_{m-1}] = & \mu Z_{m-1},\\
    \mathbb{E}[Z_m^2 | Z_{m-1}] = & \mu^2 Z^2_{m-1} + Z_{m-1}\sigma^2.
\end{align*}
Taking expected values on both sides and iterating and noting $Z_0=1$, we have
$
\mathbb{E}[Z_m] = \mu^m,
$
and
\begin{align*}
\mathbb{E}[Z_m^2]  = \mu^{2m} + \sigma^2 \sum_{j=m-1}^{2m-2} \mu^j
\leq  \mu^{2m} + \sigma^2\frac{\mu^{2m-2}}{1-\mu^{-1}}.
\end{align*}

By the Paley--Zygmund inequality, the probability that the branching process survives to iteration $m$ is
\[
\mathbb{P}\{Z_m\geq 1\}\geq 
\frac{\mathbb{E}[Z_m]^2}{\mathbb{E}[Z_m^2]}
\geq \frac{\mu^{2m}}{\mu^{2m}+\sigma^2\frac{\mu^{2m-2}}{1-\mu^{-1}}}
= \frac{\mu^2-\mu}{\mu^2-\mu+\sigma^2}.
\]
Take $m\rightarrow\infty$ to finish the proof.

\end{proof}

\subsubsection{Cycle construction}
We now provide a guarantee on the output of Algorithm \ref{alg:cycle-construction}. \prettyref{lmm:Mbad} then follows as a simple corollary of the following result. 
\begin{lemma}\label{lemma:cycle-construction}
Let $\epsilon >0$ be such that $k\lambda \geq 1+\epsilon$, and recall that $\gamma = \tfrac{\epsilon}{10(1+\epsilon)}$. Let $\mathcal{C}$ be the output of Algorithm \ref{alg:cycle-construction} on input $(G, \ell, d)$, where $\ell =  \frac{2^{13} \log(32 e) k^2}{\lambda^2 \gamma^2} \alpha$, and $d = \frac{2^{11} \log (32 e) k}{\lambda \gamma} \alpha$, for $\alpha \geq 1$ sufficiently large. Then there exist constants $c_1, c_2> 0$ such that $\mathcal{C}$ contains at least $e^{c_1 n}$ cycles of length at least $c_2 n$, with probability $1 - e^{-\Omega(n)}$, for any realization of $H^*$. 
\end{lemma}
To prove Lemma \ref{lemma:cycle-construction}, we will reduce to the problem of finding large cycles in a random bipartite graph with a perfect matching, using the following key result which we record for completeness.
\begin{lemma}\cite[Lemma 7]{Ding2023}\label{lemma:bipartite-ER} 
Let $G$ be a bi-colored bipartite graph on $[m] \times [m]'$ whose $m$ red edges are defined by a perfect matching, and blue edges are generated from a bipartite Erd\H{o}s--R\'enyi graph with edge probability $\frac{D}{m}$. If $m \geq 525$ and $D \geq 256 \log(32 e)$, then with probability at least $1-\exp\left(-\frac{Dm}{2^{14}}\right)$, $G$ contains $\exp\left(\frac{m}{20}\right)$ distinct alternating cycles of length at least $\frac{3m}{4}$.
\end{lemma}

\begin{proof}[Proof of Lemma \ref{lemma:cycle-construction}]
Let $E^{\star} \subset [n]^2$ be the output of Algorithm \ref{alg:matching} on input $(G,2\gamma n/k)$, so that $|E^{\star}| = 2\gamma n/k$. Next, let $\mathcal{T}$ be the output of Algorithm \ref{alg:tree-construction} on input $(G, E^{\star}, s)$. Let $E_1$ be the event that $\mathcal{T} = \{T_i = (L_i, R_i)\}_i$ contains at least $K_1$ two-sided trees with at least $\ell$ red vertices in each subtree, where
\[K_1 =\frac{K\epsilon^2}{2(\epsilon+2k)^2} = \frac{\gamma n \epsilon^2}{4(2\ell+k)k(\epsilon+2k)^2}.\]
By Theorem \ref{theorem:enough-trees}, we have $\mathbb{P}(E_1 | H^*) = 1 - e^{-\Omega(n)}$. On the event $E_1$, assume without loss of generality that $|L_i|, |R_i| \geq \ell$ for all $i \in \{1, 2, \dots, K_1\}$.

Our next goal is to characterize the bipartite graph $\overline{G}$ constructed in Algorithm \ref{alg:cycle-construction}, on the event $E_1$. A first observation is that the blue edges between trees and tree-facing vertices are independent of the edges between linking vertices. Indeed, this independence is the reason for the five-edge linking construction. Next, we find a lower bound on the probability that a given left tree $L_i$ is connected to a right tree $R_j$ by a five-edge construction. Suppose that $L_i$ connects to $a$ edges among $E_L^{\star}$ and $R_j$ connects to $b$ edges among $E_R^{\star}$. In that case, there are $a\cdot b$ pairs of linking edges that could be used to complete a five-edge connection between $L_i$ and $R_j$, so that the probability that $L_i$ and $R_j$ are connected by a five-edge construction is
\begin{equation}
1 - \left(1-\frac{\lambda}{n}\right)^{a b} \geq \frac{\lambda ab}{2n}, \label{eq:edge-prob}    
\end{equation}
where the inequality holds for $n$ sufficiently large (using $(1-x)^y \leq 1-\frac{xy}{2}$ for $0 \leq x \leq \frac{1}{y}$). 

Intuitively, if many trees among $\{L_i\}$ and $\{R_j\}$ are blue-connected to many tree-facing endpoints, then the five-edge construction should produce many long cycles. Therefore, we would like to show that many trees connect to some large constant number of tree-facing endpoints. At the same time, we need to control for collisions; that is, when two trees connect to the same tree-facing endpoint, since in those cases we lose the requisite independence in the five-edge construction. For this reason, the construction of $\overline{G}$ considers trees in sequence, and avoids such collisions by design (see Step \ref{step:collision}).

For some $c > 0$ to be determined, let $\mathcal{E}$ be the event that Algorithm \ref{alg:cycle-construction} identifies at least $cn$ trees $(L_i, R_i)$ which are both blue-connected to $d$ unmarked tree-facing endpoints. We will show that $\mathbb{P}(\mathcal{E} \mid H^*) = 1-e^{-\Omega(n)}$. To this end, let $X_i$ be the number of unmarked tree-facing endpoints that are blue-connected to $L_i$, and let  $Y_i$ be the number of unmarked tree-facing endpoints that are blue-connected to $R_i$. Define independent random variables $\tilde{X}_i, \tilde{Y}_i \sim \text{Bin}\left(\ell \left(\frac{\gamma n}{k}-d - cn d \right) ,\frac{\lambda}{n}\right)$. We claim that 
\[\prob{\mathcal{E}^c \mid H^*} \leq \prob{\sum_{i=1}^{K_1} \mathbbm{1}\{\tilde{X}_i \geq d, \tilde{Y}_i \geq d\} \le cn }.\] 
To see this, suppose $\calE^c$ holds. Then Algorithm \ref{alg:cycle-construction} identifies at most $cn$ trees $(L_i, R_i)$. Therefore, for each $i \in [K_1]$, there are at least $\frac{\gamma n}{k}-d - cn d$ tree-facing vertices that are not yet connected to any tree. Hence, $X_i$ and $Y_i$'s stochastically dominate $\tilde{X}_i$ and $\tilde{Y}_i$, respectively. 
It follows that 
\[\prob{\mathcal{E}^c \mid H^*} 
=\prob{\sum_{i=1}^{K_1} \mathbbm{1}\{X_i \geq d, Y_i \geq d\} \le cn \mid H^*}
\leq \prob{\sum_{i=1}^{K_1} \mathbbm{1}\{\tilde{X}_i \geq d, \tilde{Y}_i \geq d\} \le cn }.\]

Let $Z = \sum_{i=1}^{K_1} \mathbbm{1}\{\tilde{X}_i \geq d, \tilde{Y}_i \geq d\}$. Observe that $\mathbb{E}[\tilde{X}_i]= \mathbb{E}[\tilde{Y}_i] = \ell\left(\frac{\gamma n}{k}-d - cn d\right) \frac{\lambda}{n}$. We will set $d$ so that we have 
\begin{equation}
d \leq \left \lfloor \mathbb{E}[\tilde{X}_i]  \right \rfloor = \left \lfloor \mathbb{E}[\tilde{Y}_i]  \right \rfloor. \label{eq:requirement1}    
\end{equation}
Then, using properties of the binomial distribution and independence of $\tilde{X}_i$ and $\tilde{Y}_i$, we have that 
\[\prob{\tilde{X}_i \geq d, \tilde{Y}_i \geq d} \geq \frac{1}{4}.\]
It follows that $\mathbb{E}[Z] \geq \frac{K_1}{4}$, and for any $\delta \in (0,1)$, a Chernoff bound yields
\[\prob{Z \geq (1-\delta)\frac{K_1}{4}} \geq e^{-\frac{\delta^2 K_1}{8}} = e^{-\Theta(n)}.\]
By requiring 
\begin{equation}
cn \leq (1-\delta) \frac{K_1}{4}, \label{eq:requirement2}    
\end{equation}
we see that $\mathbb{P}(\mathcal{E}) = 1 - e^{-\Omega(n)}$. 

It follows that on the event $\mathcal{E}$, the graph $\overline{G}$ can be coupled to a bi-colored bipartite graph $H$ with at least $cn$ vertices on each side, a perfect (red) matching, and random blue edges which exist with probability $\frac{\lambda d^2}{2n}$ independently, due to \eqref{eq:edge-prob} (and independently of $\mathcal{E}$). 
To apply Lemma \ref{lemma:bipartite-ER}, we need to verify 
\[\frac{\lambda d^2}{2n} \cdot cn \geq 256 \log(32 e).\]
We simply let $c = \frac{512 \log (32 e)}{\lambda d^2}$ to ensure the above. It remains to show \eqref{eq:requirement1} and \eqref{eq:requirement2}. 
To show \eqref{eq:requirement1}, observe that
\begin{align}
\left \lfloor \mathbb{E}[\tilde{X}_i]  \right \rfloor &\geq \frac{1}{2}\mathbb{E}[\tilde{X}_i] \nonumber \\
&= \frac{\ell}{2}\left(\frac{\gamma n}{k}-d - cn d\right) \frac{\lambda}{n} \nonumber\\
&\geq \frac{\ell}{2}\left(\frac{\gamma n}{k} - 2cn d\right) \frac{\lambda}{n}\nonumber\\
&= \frac{\ell \lambda }{2}\left(\frac{\gamma }{k} - 2c d\right)\nonumber\\
&=\frac{\ell \lambda }{2}\left(\frac{\gamma }{k} - \frac{1024 \log (32 e)}{\lambda d}\right). \label{eq:d-condition}
\end{align}
where the inequality holds for $n$ sufficiently large. Recall the definitions of $\ell$ and $d$ in the lemma statement. 
Then \eqref{eq:d-condition} is lower-bounded by
\begin{align*}
\frac{\ell \lambda}{2}\left(\frac{\gamma}{k} - \frac{1024 \log (32 e)}{\lambda } \cdot \frac{\lambda \gamma}{2^{11} \log (32 e) k} \right)&= \frac{\ell \lambda}{2} \cdot \frac{\gamma}{2k}=  \frac{2^{11} \log(32 e) k}{\lambda \gamma} \alpha = d,
\end{align*} 
hence verifying \eqref{eq:requirement1}.

Finally,
\begin{align*}
cn &= \frac{512 \log(32e)}{\lambda d^2}n \leq \frac{\gamma n \epsilon^2}{16(2\ell+k)k(\epsilon+2k)^2} = \frac{K_1}{4},
\end{align*}
where the inequality holds for $\alpha$ sufficiently large.

By Lemma \ref{lemma:bipartite-ER}, we conclude that with probability at least $1 - \exp\left(-\frac{\lambda d^2}{2^{15}n} (cn)^2 \right) - e^{-\Omega(n)} = 1 - e^{-\Omega (n)}$, the graph $\overline{G}$ contains at least $\exp\left(\frac{cn}{20} \right)$ distinct alternating cycles of length at least $\frac{3cn}{4}$, conditioned on any realization of $H^*$.
\end{proof}

The proof of Lemma \ref{lmm:Mbad} now follows directly.
\begin{proof}[Proof of Lemma \ref{lmm:Mbad}]
Each distinct alternating cycle $C \in \mathcal{C}$ induces a distinct $k$-factor. Furthermore, an alternating cycle of length $\delta n$ induces a $k$-factor $H$ satisfying $\ell(H, H^*) = \frac{2\delta}{k}$.
Taking $c_1, c_2$ from Lemma \ref{lemma:cycle-construction}, we therefore let $\delta \leq c_0 := c_2$. Lemma \ref{lemma:cycle-construction} implies that $|\Mbad| \geq e^{c_1 n}$ with probability $1-e^{-\Omega(n)}$, for any realization of $H^*$.
\end{proof}

\begin{remark}
Our strategy of constructing trees and linking them via a sprinkling procedure is very similar to \cite{Ding2023}. However, there are a few differences. First, recall that the model considered by \cite{Ding2023} is a planted matching where the background graph is a bipartite \ER random graph, while our background graph is unipartite. The tree construction process is essentially the same, though we need to take care to ensure that every blue vertex in the tree is followed by $k$ red edges. We modify the way in which trees are linked, since our graph is not bipartite, though it is convenient to designate reserved edges as being either ``left'' or ``right.'' Our choice to name the endpoints of the reserved edges as ``tree-facing'' or ``linking'' is similarly for ease of analysis.

As in \cite{Ding2023}, we reduce our problem of connecting the trees into cycles to exhibiting a well-connected bi-colored bipartite graph with the trees as nodes, where blue edges are independent and red edges form a perfect matching. However, we follow a different path to constructing the desired bipartite graph, specifically in the way we avoid collisions.  
While our approach identifies trees to include in a sequential manner, the approach of \cite{Ding2023} instead computes the number of non-colliding edges that each tree is connected to, and argues that many trees (a suitable linear number) are connected to many (a suitably large constant) number of non-colliding edges. 
\end{remark}

\section{Proofs for partial recovery}
We first present the proof for the positive direction of partial recovery. 
\begin{proof}[Proof of \prettyref{thm:partial}] Observe that if $u$ has degree $k$ in $G$, then all edges incident to $u$ must be planted. It follows that the edges $(u,v)$ contributing to $\hat{H} \triangle H^*$ are such that neither $u$ nor $v$ has degree $k$.

If $u$ is isolated in the background graph $G_0$, then $u$ will have degree $k$ in $G$. Letting $X$ be the number of isolated vertices in $G_0$, we see that
\[\frac{|\hat{H} \triangle H^*|}{|H^*|} \leq \frac{|H^*| - \frac{1}{2} k X}{|H^*|} = 1 - \frac{X}{n}.\]
Here, the factor of $1/2$ accounts for the possibility that both endpoints of a given edge have degree $k$ in $G$. Since each vertex is isolated in $G_0$ with probability $(1-p)^{n-1} =\left(1-\frac{\lambda}{n}\right)^{n-1} \geq e^{-2\lambda}$ (the last inequality is due to $1-x \ge e^{-2x}$ for $0 \le x \le 1/2$), 
it follows that $\mathbb{E}[X] \geq e^{-2\lambda} n$ and 

\[\mathbb{E}\left[\ell(\hat{H}, H^*) \mid H^* \right] \leq 1- \expect{X}/n  \le 1 - e^{-2\lambda}.\]
Moreover, we can derive that
$\expect{X^2} = n(1-p)^{n-1}+n(n-1)(1-p)^{2n-3}$ and so
\[\text{Var}(X) = n(1-p)^{n-1}+n(n-1)(1-p)^{2n-3} - n^2(1-p)^{2(n-1)} = O(n).\]
Thus, by Chebyshev's inequality, we get that 
$$
\prob{\ell(\hat{H}, H^*) \le 1-  \frac{1}{2} e^{-2\lambda} \mid H^* }
\ge \prob{X \ge \frac{1}{2} \expect{X}  } \ge 1- \frac{4 \var(X) }{ \expect{X}^2 } \ge 1- O(1/n).
$$
Moreover, observe that all edges incident to degree-$k$ vertices in $G$ are included in every $k$-factor in $G$. Therefore, all $k$-factors in $G$
agree on at least $Xk/2$ edges. Since with probability at least $1-O(1/n)$, $X \ge \expect{X}/2 \ge e^{-2\lambda} n /2$, it follows that 
all $k$-factors in $G$
agree on at least $e^{-2\lambda} n k/4$ edges
and in particular, share at least a $e^{-2\lambda}/2$ fraction of their edges with $H^*$. 
\end{proof}

Next, we present the proof of~\prettyref{thm:nothing}, the negative direction of partial recovery.
Recall from~\prettyref{eq:posterior-tail} that while a random draw $\tilde{H}$ from the posterior distribution~\prettyref{eq:posterior_distribution} may not minimize the reconstruction error, its error is at most twice the minimum.
Thus, it suffices to analyze the posterior sample $\tilde{H}$, which relies on the following two lemmas. The first is a variation of \cite[Lemma 3.9]{Mossel2023}, which provides a high-probability lower bound on the total number of $k$-factors in the observed graph $G.$
\begin{lemma}\label{lemma:nothing-helper1}
Let $Z(G)=|\calH(G)|$ denote the number of $k$-factors contained in graph $G.$
 Let $\mathbb{Q}$ denote the distribution of \ER random graph $\calG(n,p)$ and $G_0 \sim \mathbb{Q}$. Then for any $\epsilon>0$, it holds that 
\begin{align}
\prob{Z(G) \le\epsilon  \mathbb{E}_{G_0
\sim \mathbb{Q}}[Z(G_0)] }\le \epsilon.
\label{eq:nothing-helper1}
\end{align}
\end{lemma} 
We remark that the expectation in $\mathbb{E}_{G_0 \sim \mathbb{Q}}[Z(G_0)]$ is taken over the distribution of the purely \ER random graph, while the probability in~\prettyref{eq:nothing-helper1} is taken over the distribution of the planted $k$-factor model. The proof of Lemma~\ref{lemma:nothing-helper1} follows from $\mathbb{P}(G)/\mathbb{Q}(G)= Z(G)/\mathbb{E}_{G_0\sim \mathbb{Q}}[Z(G_0)]$ and a simple change of measure. 
Note that the result in Lemma~\ref{lemma:nothing-helper1} is an instance of the so-called ``planted trick'', a technique first introduced in the study of random constraint satisfaction problems~\cite{achlioptas2008algorithmic}, and more recently employed to establish ``nothing'' results in statistical inference problems, such as group testing~\cite{coja2022statistical}; see also \cite{Mossel2023}. 

The next lemma bounds the expected number of $k$-factors in the observed graph $G$ that share $\ell$ common edges with $H^*.$
\begin{lemma}\label{lemma:nothing-helper2}
For all $\ell \in \{0,1, \dots, kn/2\}$, let 
$$
Z_\ell(H^*, G) = \sum_{H \in \mathcal{H}: |H \cap H^*| =\ell}  \indc{\text{$H$ 
 is a $k$-factor in $G$}}. 
$$
It holds that
\begin{align*}
\expect{Z_\ell(H^*, G)} 
\le \left( nkp \right)^{kn/2-\ell}.
\end{align*}
\end{lemma}

With Lemma~\ref{lemma:nothing-helper1} and Lemma~\ref{lemma:nothing-helper2}, we are ready to bound the reconstruction error of the posterior sample $\tilde{H}$.  
\begin{proof}[Proof of Theorem \ref{thm:nothing}]
It suffices to prove that 
\begin{align}
\prob{ \ell(H^\ast, \tilde{H}) \le  2(1-\delta) } = \prob{ | H^\ast \cap \tilde{H}|  \ge \delta n k}
\le  3 (k^2/c)^{-nk/4}. \label{eq:posterior_desired_bound}
\end{align}
for $\delta= \frac{\log(k^2/c)}{\log(nkp)}$. Observe that 
\begin{align*}
\prob{ | H^\ast \cap \tilde{H}|  \ge \delta n k}
=\expect{\mu_G\left( \{ H: | H^\ast \cap H| \ge \delta nk \right)}
= \expect{ \frac{1}{Z(G)} \sum_{\ell \ge \delta nk} Z_\ell(H^*, G)}. 
\end{align*}
Lemmas \ref{lemma:nothing-helper1} and \ref{lemma:nothing-helper2} imply that for any $\epsilon > 0$ (possibly depending on $n$)
\begin{align*}
& \expect{ \frac{1}{Z(G)} \sum_{\ell \ge \delta nk} Z_\ell(H^*, G)} \\
& = \expect{ \frac{1}{Z(G)} \sum_{\ell \ge \delta nk} Z_\ell(H^*, G) \indc{Z(G) > \epsilon  \mathbb{E}_{\mathbb{Q}}(G)}}
+ \expect{ \frac{1}{Z(G)} \sum_{\ell \ge \delta nk} Z_\ell(H^*, G) \indc{Z(G) \le \epsilon  \mathbb{E}_{\mathbb{Q}}(G)}} \\
& \le \expect{ \frac{1}{\epsilon \mathbb{E}_{\mathbb{Q}}[Z(G)]} \sum_{\ell \ge \delta nk} Z_\ell(H^*, G) }
+\prob{ Z(G) \le \epsilon  \mathbb{E}_{\mathbb{Q}}[Z(G)]} \\
& \le \frac{1}{\epsilon \mathbb{E}_{\mathbb{Q}}[Z(G)]} \sum_{\ell \ge \delta nk} (nk p)^{nk/2-\ell} +\epsilon.
\end{align*}
Note that 
$
\mathbb{E}_{\mathbb{Q}}[Z(G)] = M p^{nk/2},
$
where $M$ is the number of labeled $k$-factors in the complete graph. It is known (cf.~\cite[Corollary 2.17]{bollobas2001random})
that
$$
M \sim \sqrt{2} e^{-(k^2-1)/4} \left( 
\frac{k^{k/2} }{e^{k/2} k! } \right)^n  n^{nk/2}.
$$
Therefore,
$
\mathbb{E}_{\mathbb{Q}}[Z(G)]  \ge (c np/k)^{nk/2}
$
for some universal constant $c<1.$ Hence,
\begin{align}
\expect{ \frac{1}{Z(G)} \sum_{\ell \ge \delta nk} Z_\ell(H^*, G)}
\le \frac{1}{\epsilon}  (k^2/c)^{nk/2} \sum_{\ell \ge \delta nk} (nkp)^{-\ell} +\epsilon
\le \frac{2}{\epsilon}  (k^2/c)^{nk/2} (nkp)^{-\delta nk} + \epsilon. \label{eq:posterior}
\end{align}
Setting $\epsilon^2= (k^2/c)^{nk/2} (nkp)^{-\delta nk} $ and recalling $\delta = \frac{\log(k^2/c)}{\log(nkp)}$, we have
\begin{align*}
\frac{2}{\epsilon}  (k^2/c)^{nk/2} (nkp)^{-\delta nk}  + \epsilon 
\le 3 (k^2/c)^{nk/4} (nkp)^{-\delta nk/2} 
\le 3 (k^2/c)^{-nk/4}.
\end{align*}
Substituting the last display into \eqref{eq:posterior} yields the desired bound~\prettyref{eq:posterior_desired_bound}.

Moreover, by Markov's inequality, 
\begin{align*}
\prob{\mu_G\left( \{ H: | H^\ast \cap H| \ge \delta nk \right) \ge (k^2/c)^{-nk/8}
}
&\le \frac{\expect{\mu_G\left( \{ H: | H^\ast \cap H| \ge \delta nk \right)} }{(k^2/c)^{-nk/8}}\\
&\le 3(k^2/c)^{-nk/8}  
\end{align*}
In other words, with probability at least $1- 3(k^2/c)^{-nk/8}$, at least $1-(k^2/c)^{-nk/8}$
fraction of $k$-factors in graph $G$ share at most a $2\delta$-fraction of their edges with $H^*$.
\end{proof}
We now provide the proofs of Lemmas \ref{lemma:nothing-helper1} and \ref{lemma:nothing-helper2}. 

\begin{proof}[Proof of Lemma \ref{lemma:nothing-helper1}]
Note that 
\begin{align*}
\mathbb{P}(G)  = \sum_{H\in \calH} \mathbb{P}(H^*=H) \mathbb{P}(G \mid H^*=H )  = \frac{1}{|\calH|} \sum_{H \in \calH }
\prod_{e \in H } \indc{G_e=1}
\prod_{e \notin H} p^{G_e} (1-p)^{1-G_e}
\end{align*}
and $\mathbb{Q}(G)= \prod_{e \in \binom{[n]}{2}} p^{G_e} (1-p)^{1-G_e}.$ Therefore,
$$
\frac{\mathbb{P}(G)}{\mathbb{Q}(G)}
= \frac{1}{|\calH|} \sum_{H \in \calH }
\prod_{e \in H} \frac{\indc{G_e=1}}{p^{G_e} (1-p)^{1-G_e}}
=\frac{1}{|\calH|} \sum_{H \in \calH }
\prod_{e \in H} \frac{\indc{G_e=1}}{p^{G_e}} 
= \frac{Z(G)}{|\calH| p^{nk/2}} 
= \frac{Z(G)}{\mathbb{E}_{G_0\sim \mathbb{Q}}[Z(G_0)]}.
$$
Therefore,
\begin{align*}
\prob{Z(G) \le \epsilon  \mathbb{E}_{G_0 \sim \mathbb{Q}}[Z(G_0)] }
& =\mathbb{E}_{\mathbb{Q}}\left[ \frac{\mathbb{P}(G)}{\mathbb{Q}(G)} \indc{Z(G) \le \epsilon  \mathbb{E}_{G_0\sim \mathbb{Q}}[Z(G_0)]}\right] \\
& =\mathbb{E}_{\mathbb{Q}}\left[  \frac{Z(G)}{\mathbb{E}_{G_0\sim \mathbb{Q}}[Z(G_0)]} \indc{Z(G) \le \epsilon  \mathbb{E}_{G_0 \sim \mathbb{Q}}[Z(G_0)]}\right]
\le \epsilon.
\end{align*}
\end{proof}

\begin{proof}[Proof of Lemma \ref{lemma:nothing-helper2}]
By definition,
\begin{align*}
\expect{Z_\ell(H^*, G)} 
&=\sum_{H\in \mathcal{H}: | H^\ast \cap H| =\ell}
\prob{\text{$H$ is a $k$-factor in $G$}}\\
&=\sum_{H\in \mathcal{H}: | H^\ast \cap H| =\ell}
p^{kn/2-\ell}
\le (nkp)^{kn/2-\ell},
\end{align*}
where the last inequality follows from~\prettyref{eq:M_ell_bound} in~\prettyref{lmm:enumeration} that 
$
\left|\{H \in \mathcal{H}: | H^\ast \cap H|\}\right| \le (nk)^{nk/2-\ell}.
$
\end{proof}

\section{Proofs for iterative pruning algorithm}

In this section, we prove~\prettyref{thm:size}, the performance guarantee for the iterative pruning algorithm. 

\subsection{Proof of Error Upper Bound}
In this subsection, we show that \[\prob{e \in C_n \mid H^* } \le (1-\rho)^2+o(1).\] 
We need to appropriately define the local neighborhood and the branching process. 
\begin{definition}[Alternating $t$-neighborhood]~\label{def:neighorhood}
Given a planted edge $e$ and integer $t \ge 0,$  we define its 
alternating $t$-neighborhood $G_e^t$
as the subgraph formed by all alternating paths of length no greater than $t$ starting from edge $e$ (not counting $e$).  Let $\partial G_e^t$ denote the set of nodes from which the shortest alternating path to $e$ has exactly $t$ edges (not counting $e$). 
\end{definition}

\begin{definition}[Alternating $t$-branching process]
Given a planted edge $e$ and integer $t \ge 0$, we define an 
alternating $t$-branching process $T_e^t$ recursively as follows. 
Let $T_e^0$ be the single edge $e$ 
and assign its two endpoints to $\partial T_e^0$. 
For all $0 \le s \le t-1$, if $s$ is even (resp.\ odd), for each vertex $u$ in $\partial T_e^s$, we include an independent $\Pois(\lambda)$ number of blue edges (resp.\ a fixed $k$ number of red edges) $(u,v)$ to $T_e^{s+1}$
and include $v$ in $\partial T_e^{s+1}.$
\end{definition}

\begin{lemma}[Coupling lemma]~\label{lmm:coupling}
Suppose $ t \ge 0$ and $(2k\lambda+2)^t \log n=n^{o(1)}$ (for which $t = o(\log n)$ suffices). For any planted edge $e,$ there exists a  coupling between $G_e^{2t}$ and $T_e^{2t}$ (with an appropriate vertex mapping) such that 
$$
\prob{G_e^{2t} = T_e^{2t} } \ge 1- n^{-\Omega(1).}
$$
\end{lemma}

It is well known that the standard notion of $t$-hop neighborhood of a given vertex in an \ER random graph with a constant average degree $\lambda$  can be coupled with a Galton--Watson tree with $\Pois(\lambda)$ offspring distribution with high probability for $t=o(\log n)$, see, e.g.,~\cite[Proposition 4.2]{mossel2015reconstruction} and~\cite[Lemma 10, Appendix C]{hajek2018recovering}. \prettyref{lmm:coupling} follows from similar ideas. However, we need to properly deal with the extra complications arising from two colored edges. For instance, we may have cycles solely formed by red edges in the local neighborhood; however, this will not be included in the alternating $t$-neighborhood as per~\prettyref{def:neighorhood}.

Let $C^{2t}$ denote the event 
$$
C^{2t} = \{\left|\partial G_e^{2s-1} \right| \le 2\lambda (2\lambda k+2)^{s-1} \log n, \, 
\left| \partial G_e^{2s} \right| \le 2\lambda k (2\lambda k+2)^{s-1} \log n
, \, 
\forall 1 \le s \le t \}.
$$
The event $C^{2t}$ is useful to ensure that $[n]\backslash V(G_e^{2t})$ is large enough so that 
$\Binom(n-|V(G_e^{2t})|, \lambda/n)$ can be coupled to $\Pois(\lambda)$ with small total variational distance. The following lemma shows that $C^{2t}$ happens with high probability conditional on $C^{2(t-1)}$.

\begin{lemma}\label{lmm:C2t}
For all $t \ge 1$, 
$$
\prob{C^{2t} \mid G_e^{2(t-1)}, C^{2(t-1)} }
\ge 1- n^{- \lambda/3 },
$$
and conditional on $C^{2t}$,
$\left|  V(G_e^{2t}) \right| \le (2\lambda k +2)^{t+1} \log n$. 
\end{lemma}
\begin{proof}
In this proof, we condition on $G_e^{2(t-1)}$ such that the event $C^{2(t-1)}$ holds. 
Then
$|\partial G_e^{2(t-1)}| \le 2\lambda k (2\lambda k+2)^{t-2} \log n \le  (2\lambda k+2)^{t-1} \log n$. 
For any $u \in \partial G_e^{2(t-1)},$
let $B_u$ denote the number of blue edges connecting $u$ to vertices in $[n]$.
Note that since $u \in \partial G_e^{2(t-1)}$, the shortest alternating path from $e$ to $u$ has $2(t-1)$ edges. Thus, $u$ does not connect to any vertex in $\partial G_e^{2s}$ via a blue edge for all $0 \le s \le t-2.$
Thus $\{B_u\}$'s are stochastically dominated by 
 i.i.d.\ $\Binom(n,\lambda/n)$. 
It follows that $|\partial G_e^{2t-1}|$ is stochastically dominated by
$$
X \sim \Binom\left(   (2k\lambda+2)^{t-1} n \log n , \, \frac{\lambda}{n} \right). 
$$
Note that $\expect{X} = (2k\lambda+2)^{t-1} \lambda \log n
\ge \lambda \log n$ for all $t \ge 1.$
Applying the Chernoff bound for the binomial distribution,  we get
$$
 \prob{ X \ge 2\expect{X}} \le \exp\left(-\expect{X}/3\right)
\le n^{-\lambda/3 }.
$$
 Moreover, for each $u \in \partial G_e^{2t-1}$, let $R_u$ denote the number of incident red edges connecting to vertices in $[n] \backslash V(G_e^{2t-1}).$ Then $R_u \le k$. 
Thus,
$|\partial G_e^{2t}| \le k | \partial G_e^{2t-1}|.$
Hence,
\begin{align*}
\prob{C^{2t} \mid  G_e^{2(t-1)}, C^{2(t-1)} }
& = \prob{| \partial G_e^{2t-1}| \le  2\lambda (2k\lambda+2)^{t-1} \log n \mid G_e^{2(t-1)}, C^{2(t-1)}} \\
& \ge \prob{ X \le  2\lambda (2k\lambda+2)^{t-1} \log n } \\
& \ge \prob{ X \le 2\expect{X}}
\ge 1- n^{-\lambda/3 }. 
\end{align*}
Finally, conditional on $C^{2t},$
\begin{align*}
\left| V(G_e^{2t})\right|
& = |V(G_e^{0})|+  \sum_{s=1}^{t} \left(\left|\partial G_e^{2s-1}\right| + 
\left|\partial G_e^{2s}\right|\right) \\
& \le 2+ (1+k) 2 \lambda \sum_{s=1}^{t}  (2\lambda k+2)^{s-1} \log n \\
& \le 2+ 4\lambda k \frac{(2\lambda k+2)^t -1 }{2\lambda k +1}  \log n
\le 2 (2\lambda k+2)^t \log n \le (2\lambda k+2)^{t+1} \log n.
\end{align*}
\end{proof}

For each vertex $u$, let $N_u^B$ (resp.\ $N_u^R$) denote the set of neighbors  
of $u$ that are connected via a blue (resp.\ red) edge in $G$. 
For $t \ge 0$, let $E^{2t}$ denote the event 
\begin{align}
E^{2t}= \left\{ N_u^B \cap V(G_e^{2t}) = \emptyset, \forall u \in \partial G_e^{2t} \right\} \cap \left\{  N_u^B \cap N_{v}^B =\emptyset, \forall u \neq v \in \partial G_e^{2t} \right\} \label{eq:E_2t_def}
\end{align}
and $E^{2t+1}$ denote the event
\begin{align}
E^{2t+1}=
 \left\{ N_u^R \cap V(G_e^{2t+1}) = \emptyset, \forall u \in \partial G_e^{2t+1} \right\}
 \cap
\left\{  N_u^R \cap N_{v}^R =\emptyset, \forall u \neq v \in \partial G_e^{2t+1} \right\}. \label{eq:E_2t_1_def}
\end{align}

Basically, $E^{2t}$ ensures that when we grow from the $2t$-th hop neighborhood of $e$ to its $(2t+1)$-th hop neighborhood, all the added blue edges are connecting to distinct vertices in $V^{2t}$. Similarly, $E^{2t+1}$ ensures that when we grow from the $(2t+1)$-th hop neighborhood of $e$ to its $(2t+2)$-th hop neighborhood, all the added red edges are connecting to distinct vertices in $V^{2t+1}$. 
Therefore, if $E^s$ holds for all $s=1, \ldots, t$, then $G_e^t$ is a tree.

\begin{lemma}\label{lmm:EgivenC}
For any $t \ge 0$ such that $(2\lambda k+2)^t \log n =n^{o(1)},$
$$
\prob{E^{2t} \cap E^{2t+1} \mid G_e^{2t}, C^{2t} } \ge 1- n^{-1+o(1)}
$$
\end{lemma}
\begin{proof}
We first show 
$\prob{(E^{2t})^c \mid G_e^{2t}, C^{2t} } \le n^{-1+o(1)}.$ 
By the definition of $E^{2t}$ in~\prettyref{eq:E_2t_def}, we have 
\begin{align*}
(E^{2t})^c &=
\left\{ \exists u  \in \partial G_e^{2t}, w \in V(G_e^{2t}): w \in N_u^B  \right\} \\
&\quad \cup 
\left\{ \exists u \neq v \in 
\partial G_e^{2t},  w \notin V(G_e^{2t}):  w \in N_u^B, w \in N_v^B 
 \right\}.
\end{align*}
Observe that 
\begin{align*}
& \prob{\exists u  \in \partial G_e^{2t},  w \in V(G_e^{2t}): w \in N_u^B   \mid   G_e^{2t}, C^{2t}} \\
& \le \sum_{ u \in \partial G_e^{2t}}
\sum_{w \in V(G_e^{2t})} 
\prob{ w \in N_u^B  \mid   G_e^{2t},  C^{2t}}
\le \left| \partial G_e^{2t} \right|\times
\left| V(G_e^{2t}) \right|\times \frac{\lambda}{n} \\
& \le (2\lambda k+2)^t \log n \times (2\lambda k+2)^{t+1} \log n \times \frac{\lambda}{n} = n^{-1+o(1)},
\end{align*}
where the first inequality follows from the union bound,
and the second inequality holds for the following reasons. If $w \in \partial G_e^{2s}$ for $0 \le s \le t-1$, then $w \notin N_u^B$, because otherwise, the shortest alternating path from $u$ to $e$ would have length at most $2s+1$, violating the fact that $u \in \partial G_e^{2t}$; If $w \in V(G_e^{2t})\setminus \cup_{s=0}^{t-1} \partial G_e^{2s}$, then 
$\prob{ w \in N_u^B  \mid    C^{2t}} \le \lambda/n$. 
In addition,
\begin{align*}
& \prob{\exists u \neq v \in 
\partial G_e^{2t}, w \notin V(G_e^{2t}):  w \in N_u^B, w \in N_v^B \mid G_e^{2t}, C^{2t}} \\
& \le \sum_{u \neq v \in 
\partial G_e^{2t}}
\sum_{w \notin V(G_e^{2t})}
\prob{w \in N_u^B, w \in N_v^B \mid G_e^{2t}, C^{2t}}\\
& \le \left|\partial G_e^{2t} \right|^2 \times n (\lambda/n)^2\\
& \le 
(2\lambda k+2)^{2t} \log^2(n) 
 \times \lambda^2/n = n^{-1+o(1)}.
\end{align*}
Combining the last three displayed equations with a union bound yields  
\[\prob{(E^{2t})^c \mid G_e^{2t}, C^{2t} } \le n^{-1+o(1)}.\]

It remains to show 
$\prob{(E^{2t+1})^c \mid G_e^{2t}, C^{2t} } \le n^{-1+o(1)}.$ By the definition of $E^{2t+1}$ in~\prettyref{eq:E_2t_1_def}, we have
\begin{align}
(E^{2t+1})^c = & 
\left\{ \exists u  \in \partial G_e^{2t+1},  w \in V(G_e^{2t+1}): w \in N_u^R  \right\} \nonumber \\
& \cup 
\left\{ \exists u \neq v \in 
\partial G_e^{2t+1},  w \notin V(G_e^{2t+1}):  w \in N_u^R, w \in N_v^R 
 \right\}. \label{eq:E-complement}
\end{align}
Observe that the first event in \eqref{eq:E-complement} satisfies
\begin{align*}
&\left\{ \exists u  \in \partial G_e^{2t+1},  w \in V(G_e^{2t+1}): w \in N_u^R  \right\} \\
&=\left\{ \exists u' \in \partial G_e^{2t}, u  \notin V(G_e^{2t}), w \in V(G_e^{2t+1}): w \in N_u^R,
u \in N_{u'}^B \right\} \\
&=
\left\{ \exists  u' \in \partial G_e^{2t}, u  \notin V(G_e^{2t}), w \in V(G_e^{2t}):
w \in N_u^R,
u \in N_{u'}^B \right\} \\
&\quad \cup \left\{ \exists  u' \in \partial G_e^{2t}, u  \notin V(G_e^{2t}), w \in \partial G_e^{2t+1}:
w \in N_u^R, u \in N_{u'}^B \right\}\\
& = 
\left\{ \exists  u' \in \partial G_e^{2t}, u  \notin V(G_e^{2t}), w \in V(G_e^{2t}):
w \in N_u^R,
u \in N_{u'}^B \right\}\\
& \quad \cup \left\{ \exists  u' \in \partial G_e^{2t}, w' \in \partial G_e^{2t}, u  \notin V(G_e^{2t}), 
w \notin V(G_e^{2t}): w \in N_u^R,
u \in N_{u'}^B, w \in N_{w'}^B \right\},
\end{align*}
where the first equality holds because $u \in \partial G_e^{2t+1}$ if and only if $u \notin V(G_e^{2t})$ is connected to some $u' \in \partial G_e^{2t}$ via a blue edge;
the second equality holds when we decompose $V(G_e^{2t+1})$ into $V(G_e^{2t})$ and $\partial G_e^{2t+1}$; the last equality holds because $w \in \partial G_e^{2t+1}$ if and only if $w \notin V(G_e^{2t})$ is connected to some $w' \in \partial G_e^{2t}$ via a blue edge. It follows from a union bound that  
\begin{align*}
& \prob{\exists u  \in \partial G_e^{2t+1},  w \in V(G_e^{2t+1}): w \in N_u^R  \mid H^*, G_e^{2t}, C^{2t}} \\
& \le \prob{ \exists u' \in \partial G_e^{2t}, u  \notin V(G_e^{2t}), w \in V(G_e^{2t}): u \in N_{u'}^B, w \in N_u^R \mid H^*, G_e^{2t}, C^{2t}} \\
& \quad + \prob{ \exists u', w' \in \partial G_e^{2t}, u  \notin V(G_e^{2t}), w \notin V(G_e^{2t}): u \in N_{u'}^B, w \in N_{w'}^B, w \in N_u^R \mid H^*, G_e^{2t}, C^{2t} }  
\\
& \le \sum_{ u' \in \partial G_e^{2t}}
\sum_{w \in V(G_e^{2t})} 
\sum_{u \in (V(G_e^{2t}))^c \cap N_w^R} 
\prob{ u \in N_{u'}^B \mid H^*, G_e^{2t}, C^{2t} } \\
& \quad + \sum_{ u', w' \in \partial G_e^{2t}}
\sum_{w \notin V(G_e^{2t})} 
\sum_{u \in (V(G_e^{2t}))^c 
\cap N_w^R } 
\prob{ u \in N_{u'}^B, w \in N_{w'}^B \mid H^*, G_e^{2t}, C^{2t} } 
\\
& \le \left| \partial G_e^{2t} \right| 
\left| V(G_e^{2t}) \right| \frac{k\lambda}{n} + \left| \partial G_e^{2t} \right|^2  n k \left( \frac{\lambda}{n}\right)^2  \\
& \le (2\lambda k+2)^{2t+1} \log^2(n)   \frac{k\lambda}{n} + (2\lambda k+2)^{2t} \log^2(n) k \lambda^2 /n = n^{-1+o(1)}.
\end{align*}
Similarly, the second event in \eqref{eq:E-complement} satisfies
\begin{align*}
& \left\{ \exists u \neq v \in 
\partial G_e^{2t+1},  w \notin V(G_e^{2t+1}):  w \in N_u^R, w \in N_v^R 
 \right\}\\
&=\left\{ 
\exists  u', v' \in \partial G_e^{2t},
u, v \notin V(G_e^{2t}),
w \notin V(G_e^{2t+1}):
w \in N_u^R, w \in N_v^R,
u \in N_{u'}^B, v \in N_{v'}^B
\right\}.
\end{align*}
It follows that
\begin{align*}
& \prob{\exists u \neq v \in 
\partial G_e^{2t+1},  w \notin V(G_e^{2t+1}):  w \in N_u^R, w \in N_v^R \mid H^*, G_e^{2t}, C^{2t} } \\
& \le \sum_{u',v' \in \partial G_e^{2t}}
\sum_{w \in [n]} \sum_{u,v \in (V(G_e^{2t}))^c \cap N_w^R} 
\prob{u \in N_{u'}^B, v \in N_{v'}^B \mid   H^*, G_e^{2t},  C^{2t}} \\
& \le \left| \partial G_e^{2t}\right|^2 n k^2 (\lambda/n)^2\\
& \le (2\lambda k+2)^{2t} \log^2(n) \times  n k^2 (\lambda/n)^2 = n^{-1+o(1)}.
\end{align*}
Hence, recalling \eqref{eq:E-complement}, we deduce that
\begin{align*}
&\prob{(E^{2t+1})^c \mid H^*, G_e^{2t}, C^{2t}} \\
& \le \prob{\exists u  \in \partial G_e^{2t+1},  w \in V(G_e^{2t+1}): w \in N_u^R  \mid H^*, G_e^{2t}, C^{2t}} \\
& \quad + \prob{\exists u \neq v \in 
\partial G_e^{2t+1},  w \notin V(G_e^{2t+1}):  w \in N_u^R, w \in N_v^R \mid H^*, G_e^{2t}, C^{2t} } \\
& \le n^{-1+o(1)}.
\end{align*}
Further taking an average over $H^*$, we get that
$\prob{(E^{2t+1})^c \mid G_e^{2t}, C^{2t}}
\le n^{-1+o(1)}$.
\end{proof}

We are ready to construct the coupling and prove~\prettyref{lmm:coupling}. 

\begin{proof}[Proof of~\prettyref{lmm:coupling}]
We need the following bound on the total variation distance between
the binomial distribution and a Poisson distribution with approximately the same
mean:
\begin{align}
\mathrm{TV}\left( \Binom(m,p), \Pois(\lambda) \right) \le mp^2 + \psi(\lambda-mp), \label{eq:binom_poisson_coupling}
\end{align}
where $\psi(x)=1 - e^{-|x|} \le |x|$. The bound follows from $\mathrm{TV}\left( \Binom(m,p), \Pois(mp) \right) \le mp^2$ (see, e.g.~\cite[eq.(55)]{hajek2018recovering}), and the fact that $\mathrm{TV}\left( \Pois(\lambda),  \Pois(\lambda')\right) \le 1 - e^{-(\lambda - \lambda')}$ as $\text{Pois}(\lambda)$ has the same distribution as $\text{Pois}(\lambda') + \text{Pois}(\lambda - \lambda')$ for $\lambda > \lambda'$.

We construct the coupling recursively. For the base case with $t=0$, clearly $\prob{G_e^0=T_e^0}=1.$ 

Condition on $T_e^{2t}=G_e^{2t}$ (with an appropriate vertex mapping)
and event $C^{2t}$. 
We aim to construct a coupling so that 
$T_e^{2t+1}=G_e^{2t+1}$ and $T_e^{2t+2}=G_e^{2t+2}$ with probability at least 
$1-n^{-\Omega(1)}.$

Each vertex $u$ in $\partial G_e^{2t}$
has $B_u$ number of incident blue edges connecting to vertices in $[n]\backslash V(G_e^{2t})$, where the $B_u$'s are i.i.d.\ $\Binom(n-|V(G_e^{2t})| ,\lambda/n)$.
Similarly, each vertex $u$ in $\partial T_e^{2t}$ has $\tilde{B}_u$ number of incident blue edges, where the $\tilde{B}_u$'s are i.i.d.\ $\Pois(\lambda).$
Thus, we can couple $B_u$'s to $\tilde{B}_u's$ using~\prettyref{eq:binom_poisson_coupling} and  take a union bound over $u \in \partial G_e^{2t}\equiv \partial T_e^{2t}$.
In particular,
\begin{align*}
& \prob{ B_u = \tilde{B}_u, \forall u \in \partial G_e^{2t} 
\mid G_e^{2t}= T_e^{2t}, C^{2t}} \\
& \ge 1- \left|\partial G_e^{2t}\right|
\left( \lambda^2/n + \psi \left( \lambda -(n-|V(G_e^{2t})|)  \lambda/n \right) \right) \\
& \ge
1- (2k\lambda+2)^t \log n
\left( \lambda^2/n +   (2k\lambda+1)^{t+1} \lambda/n \right) \\
& \ge 1- n^{-1+o(1)},
\end{align*}
where the second inequality holds because conditional on $C^{2t}$, $|\partial G_e^{2t}| \le (2k\lambda+2)^t \log n$ and $|V(G_e^{2t})| \le (2k\lambda+1)^{t+1} \log n$. 
Thus, we have constructed a coupling such that $B_u=\tilde{B}_u$ for all $u \in \partial G_e^{2t}$ with probability at least $1-n^{-1+o(1)}$.

Recall that if event $E^{2t}$ occurs, the set of blue edges added to $G_e^{2t+1}$ connect to distinct vertices in $[n]\setminus V(G_e^{2t})$. Thus, on event $E^{2t} \cap \{B_u = \tilde{B}_u, \forall u \in \partial G_e^{2t}\}$, there exists a one-to-one mapping from the vertices in $\partial G_e^{2t+1}$ to vertices in $\partial T_e^{2t+1}$ such that $G_e^{2t+1} = T_e^{2t+1}. $ 
Further, recall that on event $E^{2t+1}$, each vertex $u$ in $\partial G_e^{2t+1}$ has exactly $k$ incident red edges, and these red edges connect to distinct vertices in $[n]\setminus V(G_e^{2t+1})$.
Thus, on the event $E^{2t+1} \cap E^{2t} \cap \{B_u = \tilde{B}_u, \forall u \in \partial G_e^{2t}\}$, there exists  a one-to-one mapping from the vertices in $\partial G_e^{2t+2}$ to the vertices in $\partial T_e^{2t+2}$, so that $G_e^{2t+2} = T_e^{2t+2}. $ 
In conclusion, we get that
\begin{align*}
&\prob{ G_e^{2t+2}= T^{2t+2}  \mid G_e^{2t}= T^{2t}, C^{2t}} \\
&\ge 
\prob{ E^{2t+1} \cap E^{2t} \cap \{B_u = \tilde{B}_u, \forall u \in \partial G_e^{2t}\} \mid G_e^{2t}= T^{2t}, C^{2t}}\\
& \ge 
\prob{ B_u = \tilde{B}_u, \forall u \in \partial G_e^{2t} \mid G_e^{2t}= T^{2t}, C^{2t}}
- \prob{\left(E^{2t+1} \cap E^{2t}\right)^c \mid G_e^{2t}= T^{2t}, C^{2t}} \\
& \ge 1- n^{-1+o(1)},
\end{align*} 
where the last inequality holds by~\prettyref{lmm:EgivenC}, since we are assuming $(2k\lambda +2)^t \log n = n^{o(1)}$.
Moreover,
\begin{align*}
&\prob{ G_e^{2t+2}= T^{2t+2}, C^{2t+2} \mid G_e^{2t}= T^{2t}, C^{2t}} \\
& \ge \prob{ G_e^{2t+2}= T^{2t+2}  \mid G_e^{2t}= T^{2t}, C^{2t}} -(1- 
\prob{C^{2t+2} \mid
G_e^{2t}= T^{2t}, C^{2t}})\\
&\ge 1- n^{-\Omega(1)},
\end{align*}
where the last inequality holds by combining the last displayed equation with~\prettyref{lmm:C2t}.
It follows that for all $t$ satisfying $(2k\lambda+2)^t \log n=n^{o(1)}$, 
\begin{align*}
\mathbb{P}(G_e^{2t} = T_e^{2t}) &= \mathbb{P}\left(\bigcap_{s=0}^{t} G_e^{2s} = T_e^{2s} \right) \\
&\geq \mathbb{P}\left(\bigcap_{s=0}^{t} \{G_e^{2s} = T_e^{2s}, C^{2s}\} \right) \\
&= \mathbb{P}(G_e^{0} = T_e^{0})\prod_{s=1}^{t} \mathbb{P}\left(G_e^{2s} = T_e^{2s}, C^{2s} \mid \bigcap_{s'= 0}^{s-1} \{G_e^{2s'} = T_e^{2s'}, C^{2s'}\} \right)\\
&= \mathbb{P}(G_e^{0} = T_e^{0})\prod_{s=1}^{t} \mathbb{P}\left(G_e^{2s} = T_e^{2s}, C^{2s} \mid  G_e^{2(s-1)} = T_e^{2(s-1)}, C^{2(s-1)} \right)\\
&\geq 1 - t n^{-\Omega(1)} \geq 1-n^{-\Omega(1)}.
\end{align*}
Thus, we get that
$\mathbb{P}(G_e^{2t} = T_e^{2t}) \ge 1-n^{-\Omega(1)}$ for all $t$ satisfying $(2k\lambda+2)^t \log n=n^{o(1)}$. 
\end{proof}

Next, we need a key intermediate result, showing that when $G_e^{2t}=T_e^{2t}$, if either side of $T_e^{2t}$ dies within depth $2t$, then 
the root edge $e$ would be pruned by the iterative pruning algorithm and thus $e \notin C_n.$
\begin{lemma}\label{lmm:pruning}
Suppose that $G_e^{2t}=T_e^{2t}$ and either side of $T_e^{2t}$ dies out within depth $(2s)$ for 
$1 \le s \le t$. Then
$e \notin C_n.$
\end{lemma}
\begin{proof}
First, let $S$ be the side of $T_e^{2t}$ that dies within depth $2s$. Since $G_e^{2t}=T_e^{2t}$ and $S$ dies out in $2s$ steps, for any vertex $u \in \partial G_e^{2s}\cap S$, there is no incident blue (unplanted) edge. Thus, all edges incident to $u$ must be planted. Hence, the iterative pruning algorithm removes vertex $u$ and all its incident edges from the graph, and decreases the capacity of the endpoints of the removed edges. Thus, for any vertex $v \in \partial G_e^{2s-1}\cap S$, all of its $k$ incident red edges will be removed and thus its capacity will drop to $0$. Therefore, the iterative pruning algorithm continues to remove vertex $v$ together with all its incident edges. Iteratively applying the above argument shows that
the iterative algorithm removes all vertices and edges in $G_e^{2t}\cap S$ at which point the vertex of $e$ in $S$ will not have any unplanted edges left. Then the algorithm will remove $e$ and hence $e \notin C_n.$
\end{proof}

Let $\rho_t$ denote the probability that the left side of the alternating branching process dies out by depth $2t$. Then we have the following recursion from the standard branching process results (cf.~\prettyref{lmm:braching_standarnd}). 
\begin{lemma}[Extinction probability]\label{lmm:extinct_prob}
Let $\phi(x)=\exp\left(-\lambda(1-x^k)\right)$ for $x \in [0,1].$
Then $\rho_0=0,$ and
$$
\rho_t = \phi\left(\rho_{t-1}\right). 
$$
If $k\lambda \equiv c\le 1$, then $\lim_{t\to \infty}\rho_t = 1$;
If $c>1,$ then $\lim_{t\to \infty}\rho_t = \rho$, where $\rho$ is defined in \eqref{eq:rho}. 
\end{lemma}

Now, we are ready to prove 
$\prob{e \in C_n \mid H^*} \le (1-\rho)^2+o(1).$  Note that 
 \begin{align*}
 \prob{e \in C_n \mid H^*} & = \prob{G_e^{2t} = T_e^{2t}, e \in C_n} + \prob{G_e^{2t} \neq T_e^{2t},e\in C_n} \\
 & \le  \prob{ \text{both sides of $T_e^{2t}$ survive to $2t$ depth} }
 +n^{-1+o(1)} \\
 & \le (1-\rho_t)^2 + n^{-1+o(1)},
 \end{align*}
where the first inequality holds by~\prettyref{lmm:pruning}
and~\prettyref{lmm:coupling}.
For any arbitrarily small constant $\epsilon>0$, we can choose $t\equiv t(\epsilon)$ large enough so that $|\rho_t -\rho| \le \epsilon/2$ by~\prettyref{lmm:extinct_prob}
and  hence
$
\limsup_{n\to \infty} \prob{e \in C_n} \le (1-\rho)^2 + \epsilon .
$
Since $\epsilon$ is an arbitrarily small  constant, we have
$\limsup_{n\to \infty} \prob{e \in C_n}  \le  (1-\rho)^2.$

\subsection{Proof of Error Lower Bound}
In this subsection, we prove  $\prob{e \in C_n} \ge (1-\rho)^2-o(1)$.   Note that this is trivially true when $k\lambda\le 1$ as $\rho=1$. Thus it suffices to focus on $k\lambda>1$. 
\begin{lemma}\label{lmm:alternating}
 A planted edge $e$ is in the core $C_n$ if it belongs to an alternating circuit in the graph $G$.
 \end{lemma}
\begin{remark}
We remark that the reverse direction of the above lemma is not true. A planted $e$ may remain in the core even if it does not belong to any  circuit. 
\end{remark}
\begin{proof}
Consider an alternating circuit containing the planted edge $e$. If we flip the colors of the edges in the circuit (planted to unplanted and vice versa), then after flipping, the planted edges still form a valid $k$-factor.  Moreover, the output of the iterative pruning procedure is unchanged. Note that the iterative pruning procedure never makes mistakes in classifying planted and unplanted edges. Thus, it will never remove any edge on this circuit. Hence $e$ must remain in the core. 
\end{proof}

Next, we lower-bound the probability that a planted $e$ belongs to an alternating circuit. 
\begin{lemma}\label{lmm:error_lower_bound_pruning}
For any planted edge $e,$
    $$
    \liminf_{n\to \infty} \prob{\text{$e$ belongs to an alternating cycle} }  \ge (1-\rho)^2.
    $$
\end{lemma}
\begin{proof}
Let $e=(i,i')$. We build a two-sided tree $T$ containing
$e$ similarly to the impossibility proof of almost exact recovery. We then create a cycle by connecting two sides of the tree to the same reserved red edge. The steps are outlined below:
\begin{enumerate}
    \item Reserve a set of $\gamma n$ red edges using Algorithm \ref{alg:matching}, avoiding $e$ and its incident red edges, where $\gamma > 0$ is a suitably small constant. For each edge $(u,v)$ with $u < v$, call $u$ the ``left'' endpoint and call $v$ the ``right'' endpoint. (Note that we will use only one reserved red edge to complete a cycle, so we do not need any further specifications for the edges.)
    \item Based on the set of reserved edges, determine the set of available vertices $\mathcal{A}$ and the set of full-branching vertices $\mathcal{F}$.
    \item Build a two-sided tree $T = (L,R)$ from $e$ by applying Algorithm \ref{alg:tree-construction} on input $\mathcal{A}$, $\mathcal{F}$, and $\ell = \sqrt{n \log n}$ (which is the size parameter). Set $K = 1$ since only one tree needs to be constructed.
    \item Find red vertices $u \in L, v \in R$ and a reserved edge $e'$ such that $u$ is connected to the left endpoint of $e'$ and $v$ is connected to the right endpoint of $e'$. (This step essentially replaces the 5-edge construction with a 3-edge construction.)
\end{enumerate}
Observe that if the above procedure is successfully executed, then an alternating cycle is constructed in the final step.

Since the tree contains at most $2(2\ell + k) = O(n)$ vertices by Proposition \ref{prop:tree.facts} (a), the size of $\mathcal{F}$ is greater than $n - 2\gamma n$ during the tree construction process. Hence, we can couple its growth to a two-sided branching distribution with offspring distribution $k \cdot \text{Binom}(n - 2\gamma n, \lambda/n)$. As long as the branching process does not die out, which happens with probability $(1-\rho_n)^2$, then the two-sided tree has at least $\sqrt{n\log n}$ red vertices on each side.

It remains to lower-bound the probability of creating a cycle. Observe that $L$ is connected to the left endpoint of a given reserved edge $e'$ with probability at least $1- (1-\lambda/n)^{\sqrt{n \log n}} \ge (\lambda/2) \sqrt{\log n/n} 
$, where the inequality holds for all sufficiently large $n$ because $1-(1-x)^m \ge  1- \exp(-mx) \ge mx/2$ for all $mx\le 1.$
Therefore, both $L$ and $R$ are connected to $e'$ (and connected on the correct side) with probability at least 
$\lambda^2 \log n/(4n).$
It follows that $L$ and $R$ are simultaneously connected to some reserved edge with probability at least
\begin{align*}
1- \left[ 1- \lambda^2 \log n /(4n) \right]^{\gamma n}
\ge 1- \exp\left( \gamma n \lambda^2 \log n / (4n) 
\right) 
=1-\exp(-\Omega(\log n)).
\end{align*}

In conclusion, we have shown that there exists an alternating cycle containing $e$ with probability at least $(1-\rho_n)^2 (1-o(1))$. The claim follows by noting $\lim_{n\to \infty} \rho_n=\rho$ in view of~\prettyref{lmm:branching_n_convergence}.
\end{proof}

Combining~\prettyref{lmm:alternating} and~\ref{lmm:error_lower_bound_pruning}, we have shown that $\liminf_{n\to \infty} \prob{e \in C_n \mid H^*} \ge (1-\rho)^2.$

\subsection{Proof of Exact Recovery}
If $k\lambda=o(1),$ we aim to show the core $C_n$ is empty.
To this end, we provide a sufficient condition under which $C_n$ is empty. 
We first define an ``almost'' alternating cycle. 
\begin{definition}
We call a cycle $(e_1, e_2, \ldots, e_t)$ almost alternating if the edges alternate between planted and unplanted except for the last one, that is,  $e_{i}$ and $e_{i+1}$ have different colors for all $1 \le i \le t-1$, while $e_t$ and $e_1$ may have the same color. 
\end{definition}
By definition, an ``almost'' alternating cycle of even length must be completely alternating.
We now claim that if graph $G$ does not contain any ``almost'' alternating cycle, then the core $C_n$ must be empty. To prove this, suppose for the sake of contradiction that $C_n$ is non-empty. Pick any planted edge $e$ in $C_n$ and consider an alternating path $P$ starting at edge $e$ that has maximal length among all such alternating paths whose edges all lie entirely in $C_n.$
Let $u$ denote the endpoint of the path and $e'$
denote the last edge on the path incident to $u.$ 
We claim that $u$ must be incident to another edge $(u,v)$ in $C_n$, not belonging to $P$, whose color is different from that of $e'$.  
Indeed, suppose this is not the case. Then all edges incident to $u$  have either red color (in which case, the remaining capacity $\kappa_u$ equals the degree of $u$) or blue color (in which case, the capacity $\kappa_u$ is zero). In either case,  the endpoint $u$ would be removed by the pruning procedure, contradicting the fact that $u \in C_n.$
Next, we argue that $v$ 
cannot lie on the alternating path $P$; otherwise, this would create an ``almost'' alternating cycle, contradicting our standing assumption that no such cycle exists.
Hence, we can extend the alternating path $P$ by appending the edge $(u,v)$,
obtaining a strictly longer alternating path contained entirely in $C_n.$ This contradicts the maximality of $P.$ Therefore, the core 
$C_n$ must be empty.

Next, we show that if $\lambda k=o(1)$, then with high probability, the graph $G$ does not contain any ``almost'' alternating cycle. Recall that in~\prettyref{eq:exact_positve}, we have already shown that the graph $G$ does not contain alternating cycles with high probability. 
Thus, it remains to show that the graph $G$ does not contain any ``almost'' alternating cycles with odd lengths. We first enumerate the number of ``almost'' alternating cycles with $t+1$ blue edges and $t$ red edges. Suppose the vertices on the alternating cycle are given by $(v_1, v_2, \ldots, v_{2t+1})$ in order, where $(v_1, v_2)$ is a red edge,
Then we can determine the labels of $v_i$'s, where $v_i$ has at most $n$ vertex labels and $v_{i+1}$ has at most $k$ vertex labels for all odd $i$ from $1$ to $2t+1$. Thus in total, we have at most $n^{t+1}k^t$ different such ``almost'' alternating cycles. Each cycle appears with probability $(\lambda/n)^{t+1}$. Thus, the probability that $G$ contains an ``almost'' alternating cycle with $t+1$ blue edges and $t$ red edges is at most $n^{t+1} k^t (\lambda/n)^{t+1} = (\lambda k)^t \lambda$. 

Next, we consider ``almost'' alternating cycles with $t$ blue edges and $t+1$ red edges. Suppose the vertices on the alternating cycle are given by $(v_1, v_2, \ldots, v_{2t+1})$ in order, where $(v_1, v_2)$ is a red edge and $(v_{2t+1},v_1)$ is a red edge. Then we can determine the labels of the $v_i$'s, where $v_i$ has at most $n$ vertex labels and $v_{i+1}$ has at most $k$ vertex labels for all odd $i$ from $1$ to $2t$. The last vertex $v_{2t+1}$ has at most $k$ labels, as it is connected to $v_1$ via a red edge. Thus in total, we have at most $n^{t} k^{t+1}$ different such ``almost'' alternating cycles. Thus, the probability that $G$ contains an ``almost'' alternating cycle with $t$ blue edges and $t+1$ red edges is at most $n^{t} k^{t+1} (\lambda/n)^{t} = (\lambda k)^t k$. 

Combining the above two cases, we get that if $\lambda k =o(1),$  then
$$
\prob{\text{$G$ contains an ``almost'' alternating cycle of odd length}} \le \sum_{t=1}^{nk/2} (\lambda k)^t (k+\lambda) 
=o(1)
$$

Combining this with our previous claim, we get that with high probability $C_n$ is empty.

\section{Conclusions and Discussions} \label{sec:conclusions}
In this paper, we have characterized the phase transitions for recovering a $k$-factor planted in an \ER random graph $\mathcal{G}(n,\lambda/n)$, as the average degree $\lambda$ varies. Additionally, we have established algorithmic limits by analyzing a linear-time iterative pruning algorithm. Some open problems arising from this work include:
\begin{itemize}
    \item \textit{What is the minimum reconstruction error when $\lambda k = \Theta(1)$?} Theorem \ref{thm:size} shows that iterative pruning achieves a reconstruction error of $(1-\rho)^2 + o(1)$. 
    \item \textit{Recovery of specific planted graphs}: What can be said about the case where $H^*$ is a graph which is known up to isomorphism? In this paper, we have treated only the case where $H^*$ is a Hamiltonian cycle (Appendix \ref{sec:Hamiltonian}). Can we predict the qualitative nature of the phase transition for recovering a graph $H^*$, based on its graph properties?
    \item \textit{Extensions to weighted graphs}: Do similar results carry over to weighted graphs? The weighted case of a planted matching $(k=1)$ was studied by \cite{Ding2023}.
    \item \textit{Extensions to growing $k$.}
    Our current analysis assumes that $k$ is fixed and does not grow with $n$. It would be interesting to extend the results to the regime where $k$ grows with $n$. We expect an ``all-or-nothing'' phase transition to occur when $k$ grows sufficiently fast with $n.$ 
    \item \textit{Extensions to $k$-factors spanning a subset of vertices.}  
    When the planted $k$-factor spans a linear number of vertices, say $\delta n$ for a constant $\delta \in (0,1)$, and $k=2$, we expect that there will be both an exact recovery and a partial recovery regime (see \cite{gaudio2025finding}). 
    However, when planting a $k$-regular graph on a sublinear number of vertices, this phenomenon may no longer hold, and we expect a different behavior. In particular, we suspect that when the planted graph is small, it is unlikely to join with edges in the background graph to form spurious $k$-regular graphs. Therefore, we speculate a sharp ``all-or-nothing" phase transition, where the threshold coincides with when $k$-regular graphs start to emerge in the background graph. We also note that different phenomenon may arises when $k=1$. When $k=1$, the graph largely consists of isolated edges, and random sampling of the edges should achieve partial recovery in a certain regime.
\end{itemize}
Finally, we note that a very recent independent work~\cite{lee2025fundamental} has also established a phase transition for recovering certain weakly dense subgraphs $H$, where $|E(H)| =\omega(|V(H)|\log|V(H)|$,
planted in an \ER random graph. Remarkably, it demonstrates that the normalized minimum mean-squared error (MMSE) exhibits a staircase-like behavior, jumping discontinuously from 0 to 1 at critical thresholds. In contrast, our work focuses on a specific family of sparse and balanced subgraphs—namely, $k$-factors—for which the normalized MMSE is expected to rise continuously from 0 to 1.

\begin{appendix}
\section{Equivalence between Hamming Error and  Mean-Squared Error}
\label{sec:equivalence}
We can equivalently represent the hidden subgraph $H^*$ in the complete graph $K_n$ as a binary vector $\beta^* \in \{0,1\}^N$, where $N=\binom{n}{2}.$
Similarly, an estimator $\hat{H}(G)$ can be represented as $\hat{\beta}(G) \in \reals^N$, where here we allow $\hat{\beta}$ to possibly take real values. There are two natural error metrics to consider:
\begin{itemize}
\item Hamming error: $\HE(\hat{\beta})=\expect{\|\hat{\beta}-\beta^*\|_0}$;
\item Mean-squared error: $\MSE(\hat{\beta})=\expect{\|\hat{\beta}-\beta^*\|_2}$,
\end{itemize}
where $\|\cdot\|_p$ denote the $L_p$ vector norm. Note that when $\hat{\beta}$ is a binary vector, $\HE(\hat{\beta})=|\hat{H}\symdiff H^*|$. The minimum mean-squared error $\inf_{\hat{\beta}} \MSE(\hat{\beta})$ is known as $\MMSE.$

The following proposition relates the two error metrics (See, e.g. \cite[Proposition 5]{reeves2021all} and the proof therein).
\begin{proposition}\label{prop:equivalence}
It holds that 
$$
\frac{1}{4} \inf_{\hat{\beta}} \HE(\hat{\beta})\le \MMSE \triangleq \inf_{\hat{\beta}} \MSE(\hat{\beta}) \le \inf_{\hat{\beta}} \HE(\hat{\beta}).
$$
Then we have the following two claims:
\begin{enumerate}
\item The almost exact recovery in $\MSE$ 
is equivalent to the almost exact recovery in $\HE$. 
\item The partial recovery in $\HE$ implies the partial recovery in $\MSE$. 
\end{enumerate}
\end{proposition}

Note that $\inf_{\hat{\beta}} \HE(\hat{\beta})$ is achieved by the maximum posterior marginal, that is, $\hat{\beta}_e= \indc{\expect{\beta^*_e|G}\ge 1/2}$,
while $\inf_{\hat{\beta}} \MSE(\hat{\beta})$ is achieved by the posterior mean, that is, $\hat{\beta}_e=\expect{\beta^*_e|G}$.
Therefore,
\begin{align*}
\MMSE & = 
\expect{\| \beta^\ast - \expect{\beta^\ast|G}\|_2^2}  \\
& =\expect{\|\beta^\ast\|_2^2} - \expect{\iprod{\beta^\ast}{\expect{\beta^\ast|G}}} \\
&=  nk/2 - \expect{\iprod{\beta^\ast}{\tilde{\beta}}},
\end{align*}
where $\tilde{\beta}$ denotes a $k$-factor randomly sampled from the posterior distribution and the last equality holds because $\tilde{\beta}$ equals 
$\beta^*$ in distribution conditional on $G$. Therefore, $\expect{\iprod{\beta^\ast}{\beta'}}=o(nk)$ implies
$\MMSE=(1-o(1))nk/2$ and hence the impossibility of partial recovery in $\MSE.$

\section{Convergence of extinction probability of branching process}\label{app:extinct}

Consider a branching process with offspring distribution $\mu_n$ supported on the non-negative integers. 
Let $\rho_{n,t}$ denote the probability that the branching process dies out by depth $t$.
Define 
$$
\phi_n(x)=\mathbb{E}_{\xi_n \sim \mu_n}\left[x^{\xi_n}\right]
$$
for $x \in [0,1].$
One can check that 
$\phi_n$ is increasing and convex in $[0,1]$ with $\phi_n(1) =1$, $\phi_n(0)=\prob{\xi_n=0}$, 
$\phi_n'(1)=\expect{\xi_n}.$
Then we have the following standard result.
\begin{lemma}[Theorem 2.1.4 in~\cite{durrett2007random}] \label{lmm:braching_standarnd}
\begin{itemize}
    \item $\rho_{n,1}=0$ and 
$\rho_{n,t}=\phi_n(\rho_{n,t-1})$.
\item If $\expect{\xi_n} >1,$ then there is a unique fixed point $\rho_{n,\infty}$ on $[0,1)$  so that $\rho_{n,\infty}=\phi_n(\rho_{n,\infty})$. Moreover, $\rho_{n,t}$ is monotone increasing in $t$ and $\lim_{t\to \infty} \rho_{n,t} = \rho_{n,\infty}$. 
\end{itemize}
\end{lemma}

 In our problem, $\mu_n$ is $k \cdot \Binom(m_n,\lambda/n)$ where $\lim_{n\to\infty} m_n/n=\alpha.$ 
Then 
$$
\phi_n(x) = \left( 1- \frac{\lambda}{n}(1-x^k) \right)^{m}
$$
Note that $\phi_n(x)$ point-wisely converges to 
$$
\phi(x) = \mathbb{E}_{\xi \sim k\cdot \Pois(\alpha \lambda)} \left[ x^\xi\right]=\exp( -\alpha\lambda(1-x^k))
$$
for any $x \in [0,1].$

In the following, we further establish the limit of $\rho_{n,\infty}$ as $n \to \infty.$
\begin{lemma}\label{lmm:branching_n_convergence}
Suppose $k\lambda>1.$ Then $\lim_{n\to \infty} \rho_{n, \infty}=\rho_{\infty}$, where $\rho_{\infty}$ is the unique fixed point in $[0,1)$ so that 
$\rho_{\infty}=\phi(\rho_{\infty}).$
\end{lemma}
\begin{proof}
Since $\phi_n(x)$'s are continuous on $[0,1]$, it follows that $\phi_n(x)$ uniformly converges to $\phi(x)$ on $[0,1]$, that is, $\lim_{n\to \infty} 
\sup_{x \in [0,1]} |\phi_n(x)-\phi(x)|=0$. More specifically, 
we claim that 
\begin{align}
\sup_{x \in [0,1]} |\phi_n(x)-\phi(x)|
\le d_n \triangleq 
\max\left\{ e^{\lambda |m/n-\alpha| } -1,
\left| e^{m \log (1-\lambda/n) +\alpha \lambda} -1 \right|\right\}. \label{eq:claim_uniform}
\end{align}
Clearly, $\lim_{n\to \infty} d_n= 0.$
To prove~\prettyref{eq:claim_uniform}, note that by replacing $1-x^k$ with $x$, it suffices to show
$$
\sup_{x \in [0,1]} \left|\left( 1- \lambda x/n \right)^m - e^{-\alpha \lambda x} \right| \le d_n
$$
Now, 
$$
\left|\left( 1- \lambda x/n \right)^m - e^{-\alpha \lambda x} \right|
= \left| e^{-\alpha \lambda x} 
\left( e^{m \log (1-\lambda x/n) + \alpha \lambda x}-1\right)\right|
\le \left|
 e^{m \log (1-\lambda x/n) + \alpha \lambda x}-1\right|.
$$
Note that $h(x) \triangleq m \log (1-\lambda x/n) + \alpha \lambda x $ is concave in $x$. Thus, for all $x \in [0,1],$
$$
h(x) \ge \min \{h(0), h(1)\}
=\min\{0, m \log (1-\lambda /n) + \alpha \lambda\}
$$
Moreover, since $\log (1-\lambda x/n) \le -\lambda x /n$, it follows that 
$
h(x) \le -\lambda x m/n
 + \alpha \lambda x \le \lambda |m/n-\alpha|
$
for all $x \in [0,1].$
In conclusion, we get that 
\begin{align*}
\sup_{x \in [0,1]} \left|\left( 1- \lambda x/n \right)^m - e^{-\alpha \lambda x} \right|
&\le \sup_{x \in [0,1]} \left|
 e^{h(x) }-1\right|\\
&\le \max\left\{ e^{\lambda |m/n-\alpha| } -1,
\left| e^{m \log (1-\lambda/n) +\alpha \lambda} -1 \right|\right\} \triangleq d_n.
\end{align*}

Now, suppose $\alpha\lambda k>1$ and let $\phi_\infty$ denote the unique fixed point on $[0,1)$ such that $\rho_\infty=\phi(\rho_\infty)$.
 Then we prove the following claim that $\lim_{n\to \infty} \rho_{n,\infty}=\rho_\infty.$  

Note that $\phi'(\rho_\infty)<1$; otherwise, by the strict convexity of $\phi(x)$, $\phi(x)>x$ for all $x \in [\rho_\infty, 1]$, which contradicts the fact that $\phi(1)=1$. By the continuity of $\phi'(x)$, there exists a small $\epsilon>0$ such that $\gamma\triangleq \phi'(\rho_\infty+\epsilon)<1.$

Recall that $d_n=\sup_{x \in [0,1]} |\phi_n(x)-\phi(x)|$. There exists $N$ such that 
for all $n \ge N, $ $d_n\le (1-\gamma)\epsilon.$

We prove by induction that for all $n \ge N$ and all $t \ge 0,$
\begin{align}
\left|\rho_{n,t} - \rho_{t}
\right| \le \sum_{s=0}^{t-1} \gamma^s d_n. \label{eq:induction_rho}
\end{align}
Note that \prettyref{eq:induction_rho} trivially holds when $t=0$, because $\rho_{n,0}=\rho_0=0.$
Suppose \prettyref{eq:induction_rho}  holds for $t-1$. Then,  
\begin{align*}
\left|\rho_{n,t} - \rho_{t}
\right| & = \left| \phi_{n}(\rho_{n,t-1})-\phi(\rho_{t-1})\right| \\
& \le  \left| \phi_{n}(\rho_{n,t-1})-
\phi(\rho_{n,t-1}) \right| + 
\left| \phi(\rho_{n,t-1}) - \phi(\rho_{t-1})\right|\\
& \overset{(a)}{\le} d_n + \gamma \left| \rho_{n,t-1} - \rho_{t-1} \right|\\
& \overset{(b)}{\le} 
\sum_{s=0}^{t-1} \gamma^s d_n,
\end{align*}
where $(a)$ holds because by the induction hypothesis $\rho_{n,t-1}
\le \rho_{t-1} + d_n/(1-\gamma) \le \rho_{\infty} +\epsilon $ and $\phi'(x) \le \gamma$ for all $x \in [0, \rho_\infty+\epsilon]$; $(b)$ holds by the induction hypothesis. 

Thus, we have shown that ~\prettyref{eq:induction_rho} also holds for $t.$ It follows that
$$
\left|\rho_{n,t} - \rho_{t}
\right| \le \sum_{s=0}^{t-1} \gamma^s d_n \le \frac{d_n}{1-\gamma}. 
$$
Taking the limit $t\to\infty$ on the above-displayed equation, we deduce that 
$$
\left|\rho_{n,\infty} - \rho_{\infty}\right|
\le \frac{d_n}{1-\gamma}.
$$
Finally, taking the limit $n \to \infty$ and noting that $\lim_{n\to \infty} d_n=0$, we get that $\lim_{n\to \infty} \rho_{n,\infty}=\rho_\infty$.
\end{proof}

\section{Finding a planted Hamiltonian Cycle}\label{sec:Hamiltonian}
All the positive results of this paper are stated by conditioning on $H^*$
 and consequently, continue to hold even if $H^*$ is constrained to be in a known isomorphism class, such as
 $H^*$ being a Hamiltonian cycle. In contrast, the negative results do not hold automatically and we need to check them separately. 

For the impossibility of exact recovery, the statement of Theorem \ref{thm:exact_recovery} that if $k\lambda=\Omega(1)$, then $G$ contains a $k$-factor $H \neq H^*$ with probability $\Omega(1)$ is still true as stated. However, $H$ would not necessarily be isomorphic to $H^*$ so its existence would not necessarily stop us from recovering $H^*$ successfully. In the case where $H^*$ is a Hamiltonian cycle, one can salvage that argument by the following simple change. Assign the cycle $H^*$ a direction and suppose $H^*$ is given by $(v_1, \ldots, v_n, v_{n+1}\equiv v_1)$. For two  nonadjacent edges in $H^*$, say $(v_i,v_{i+1})$ and
$(v_j, v_{j+1})$ for odd $i,j$ with $1 \le i<j \le n-1,$
let $\calE_{ij}$ denote the event that
the graph
$G$ has an edge connecting $v_i, v_j$ and an edge connecting $v_{i+1}, v_{j+1}$. Under event $\calE_{ij}$, replacing the original two edges $(v_i, v_{i+1}), (v_j,v_{j+1})$ with $(v_i,v_j), (v_{i+1},v_{j+1})$ would yield an alternate Hamiltonian cycle, rendering us unable to recover $H^*$ with more than a $1/2$ probability of success. 
Note that $\prob{\calE_{ij}}=\lambda^2/n^2.$
Moreover, the events $\calE_{ij}$ are mutually independent for all odd $i,j$ with $1 \le i<j \le n-1$. 
Therefore, if $\lambda=\Omega(1)$, then $\prob{\cup_{1\le i<j \le n} \calE_{ij}} =\Omega(1).$
Thus, any algorithm attempting to exactly recover $H^*$ fails with probability $\Omega(1)$.

In order to rule out almost exact recovery of a planted Hamiltonian cycle when $\lambda \ge 1/2+\epsilon$, we argue that a random $2$-factor is reasonably likely to be a cycle.

\begin{lemma}
Let $H$ be a 
random $2$-factor on $m$ vertices. With probability at least $1/m$, $H$ is a cycle.   
\end{lemma}
\begin{proof}
There are $m!$ possible directed cycles with starting points on $m$ vertices, so there are $(m-1)!/2$ possible cycles on $m$ vertices. Meanwhile, every possible $2$-factor can be converted to the cycle decomposition of a permutation of the vertices. Specifically, we assign each of its cycles a direction and then we have the permutation map each vertex to the next vertex in its cycle. 
Since each $2$-factor has at least one cycle and each cycle has two choices of directions, 
it follows that each $2$-factor can be mapped to at least $2$ permutations. Moreover, no two different $2$-factors can yield the same permutation. 
It follows that the total number of $2$-factors on $m$ vertices is at most $m!/2$. Therefore, at least $1/m$ fraction of the $2$-factors on $m$ vertices are cycles. 
\end{proof}

Theorem \ref{thm:impossibility} says that if $\lambda \ge 1/2+\epsilon,$ then 
for any estimator $\widehat{H}$,
$$
\prob{ \ell(\hat{H}, H^*) \ge \frac{\delta}{k}  }
\ge 1- e^{-\Omega(n)}. 
$$
Let $\calC_n$ denote the set of all possible Hamiltonian cycles in the complete graph on $[n]$. It follows that 
\begin{align*}
\prob{ \ell(\hat{H}, H^*) < \frac{\delta}{k} \mid H^* \in \calC_n}
& =\frac{\prob{\ell(\hat{H}, H^*) < \frac{\delta}{k} , H^* \in \calC_n}}{\prob{H^* \in \calC_n}} \\
& \le \frac{\prob{\ell(\hat{H}, H^*) < \frac{\delta}{k}} }{\prob{H^* \in \calC_n}} \\
& \le n \prob{\ell(\hat{H}, H^*) < \frac{\delta}{k}}  \le e^{-\Omega(n)}.
\end{align*}
Finally, note that the planted $2$-factor model conditioned on $H^*$ being a Hamiltonian cycle is equivalent to the planted Hamiltonian cycle model. 
Thus, we conclude that under the planted Hamiltonian cycle model with $\lambda\ge 1/2+\epsilon$, 
for any estimator $\widehat{H}$, $\ell(H^*,\widehat{H})=\Omega(1)$ with probability at least $1-e^{-\Omega(n)}.$

The impossibility of partial recovery when $\lambda=\omega(1)$ under the planted Hamiltonian cycle model can be deduced from Theorem~\ref{thm:nothing} using the same argument as above.    
\end{appendix}

\paragraph*{Acknowledgements}
The authors thank Souvik Dhara, Anirudh Sridhar, and Mikl\'os R\'acz for helpful discussions. Thank you to the anonymous reviewers for their helpful comments, which improved the presentation. J. Gaudio is supported in part by an NSF CAREER award CCF-2440539. J. Xu is supported in part by an NSF CAREER award CCF-2144593.

\bibliographystyle{plain}
\bibliography{references}

\end{document}